\documentclass[a4paper, 11 pt, twoside]{amsart}
\usepackage{srollens-en}[2011/03/03]
\usepackage[left = 3.5cm, right = 3.5cm, headsep = 6mm, footskip = 10mm, top = 35mm, bottom = 35mm, footnotesep=5mm, headheight = 2cm]{geometry}
\usepackage[expansion=false]{microtype}
\usepackage{booktabs, multirow,  tabularx}
\usepackage[usenames, dvipsnames]{xcolor}

\usepackage{amscd}

\usepackage{tikz, tikz-cd}
\usetikzlibrary{arrows}
\usetikzlibrary{calc}
\usetikzlibrary{through}
\usetikzlibrary{decorations.pathreplacing,decorations.markings}
\tikzset{
  mid arrow/.style={postaction={decorate,decoration={
        markings,
        mark=at position .5 with {\arrow[#1]{stealth}}
      }}},
}

\usepackage{enumitem}
\setlist[description]{leftmargin=0cm,  labelindent=\parindent}

\usepackage[unicode,bookmarks, pdftex, backref = page]{hyperref}
\hypersetup{colorlinks=true,citecolor=NavyBlue,linkcolor=NavyBlue,urlcolor=Orange, pdfpagemode=UseNone, breaklinks=true}

\newtheoremstyle{special-example}
  {}
  {}
  {}
  {\parindent}
  {\bfseries}
  {:}
  { }
  {}
  \theoremstyle{special-example}
\newtheorem{example}[equation]{Example}

\DeclareMathOperator{\Stab}{Stab}

\DeclareMathOperator{\ord}{ord}

\title{Double Kodaira fibrations with small signature.}

\author{Ju A Lee}
\address{Ju A Lee \\ Department of Mathematial Sciences \\ Seoul National University \\ Seoul 151-747, Korea}
\email{jualee@snu.ac.kr}
\author{Michael L\"onne}
\address{Michael L\"onne\\ Mathematik VIII\\ Universit\"at Bayreuth\\
Universit\"atstrasse 30\\ 95447 Bayreuth\\ Germany
}
\email{michael.loenne@uni-bayreuth.de}
\author{S\"onke Rollenske}
\address{S\"onke Rollenske\\FB 12/Mathematik und Informatik\\
Philipps-Universit\"at Marburg\\
Hans-Meerwein-Str. 6\\
35032 Marburg\\
Germany}
\email{rollenske@mathematik.uni-marburg.de}

\begin{document}
\begin{abstract}
Kodaira fibrations are surfaces of general type with a non-isotrivial fibration, which
are differentiable fibre bundles. They are known to have positive signature divisible
by $4$. Examples are known only with signature 16 and more.
We review approaches to construct examples of low signature 
which admit two independent fibrations.
Special attention is paid to ramified covers of product of curves which we analyse by
studying the monodromy action for bundles of punctured curves.

As a by-product we obtain a classification of all fix-point-free automorphisms on curves
of genus at most $9$. 
\end{abstract}
 \subjclass[2010]{57R22;14J29,  14D05}
\keywords{signature, Kodaira fibration, surface bundle, monodromy}

\maketitle

\setcounter{tocdepth}{1}
\tableofcontents

\section{Introduction}	
 In \cite{chs57} Chern, Hirzebruch and Serre showed that the signature is multiplicative in fibre bundles, provided the fundamental group of the base acts trivially on the cohomology of the fibre, and asked if this condition was necessary. This was answered positively by Kodaira in \cite{kodaira67}, and independently by Atiyah in \cite{atiyah69}, who both constructed an example of the following type:
\begin{defin} A \emph{Kodaira fibration} is a holomorphic submersion (with connected fibres) of a compact complex surface $f\colon S\to B$ which is not isotrivial, that is, not all fibres are biholomorphic to each other.
\end{defin}
For such surfaces the signature is indeed positive, for example, because of the formula
\[ \sigma(S ) = 4\int_B \mu_f^*c_1(\IE)\]
of \cite{atiyah69, smith99}, where $\mu_f$ is the induced map to the moduli space of curves and $c_1(\IE)$ is the first Chern class of the Hodge bundle, which is an ample class. 
It is well known that in this situation $b= g(B)\geq2$ and the genus of the fibre is at least $3$ \cite{Kas,Meyer}.

Both Kodaira and Atiyah constructed these algebraic surfaces as ramified covers of products of curves, so that they come with a pair of Kodaira fibrations. 
\begin{defin}\label{defin: double etale Kodaira} A \emph{double Kodaira fibration} is a compact complex surface $S$ together with a finite map $f\colon S\to B_1\times B_2$ to a product of curves  such that composition with the projections onto the factors induces two (necessarily different) Kodaira fibrations $f_i\colon S\to B_i$ on $S$.

Let $D\subset B_1
 \times B_2$ be the branch divisor. We call $S$ double \'etale, if the induced projections $D\to B_i$ are both unramified coverings. We say $S$ is of \emph{graph type} if $D$ is the disjoint union of graphs of maps from $B_1$ to $B_2$.
 \end{defin}
 The definition of graph type induces a slight asymmetry in the data, but often one starts from $D$ being a disjoint union of graphs of automorphisms. 

Both the algebraic geometric and the topological side of the construction have been studied intensively. Concentrating on deformations and moduli of (double \'etale) Kodaira fibrations we have \cite{Kas, jy83, cat-roll09}, and it was also shown that being a double Kodaira fibration is determined by the homotopy type of a compact complex surface \cite[Prop.~2.5]{cat-roll09}. For usual Kodaira fibrations a similar result can be found in \cite[Thm. 13.7]{Hillman} and \cite{kotschick99}. For the geometric construction using ramified covering such as \cite{kodaira67} and \cite{atiyah69}, the monodromy group was also studied in \cite{gonzalez99}.
In 4-dimensional topology, there have been studies on the determination of possible signature and the Euler characteristic of a smooth surface bundle over a surface. In \cite{kotschick99} the inequality between the signature and the Euler characteristic was proved, while \cite{En,BDS,EKKOS,BD,Lee} addressed the minimal base genus function $b(f,n)$, defined as the smallest value of $b$ for which a smooth bundle over a surface of genus $b$ with a fibre genus $f$ and a signature $4n$ exists. In particular, we have several examples of smooth $4$-manifolds with signature $4$ which are surface bundles with small genera of fibre and base constructed using the technique called the subtraction of Lefschetz fibrations \cite{EKKOS,Lee}. However, it is unknown if the total space of these examples admits any complex structure, nor if there exists any other complex surface with signature 4 diffeomorphic to a smooth surface bundle.

There are algebraic geometric variants of the construction which do not yield double \'etale or even double Kodaira fibrations \cite{gdh91, zaal95}.

The motivation for this paper came from a question left open in \cite{cat-roll09}, where the aim (among others) was to construct Kodaira fibrations of large Chern slope ${c_1^2(S)}/{c_2(S)}$. This was done by the so-called tautological construction, which subsumes all previously known ones. In the original version it involves an \'etale pullback ``of sufficiently large degree'', which does not change the Chern slope but completely looses control over the signature and the genus of the base curve.

In this paper we first revisit and generalise the tautological construction to make the degree of the necessary pullback explicit (Section \ref{sect: effective tautological}). In order to compute it effectively, we need to study the monodromy action in fibre bundles of punctured curves, which is done in Section \ref{sect: monodromy computation}. 

We show by example (Section \ref{sect: computing}) that explicit computations are indeed possible, illustrating both algebraic and geometric ways to approach the problem.

We construct explicitly several new (and old) double \'etale Kodaira fibrations, including examples with signature $16$ (see Table \ref{tab: invariants} for an overview). We were unable to answer the question if there exists a Kodaira fibration with signature at most $12$ because  the topological complexity of a complete classification of (double \'etale) Kodaira fibrations of low signature goes beyond the scope of this paper. We illustrate this point in Section \ref{sect: virtual with small signature} and by classifying all fixed point-free automorphisms on algebraic curves of small genus (up to $9$) in Appendix \ref{sect: fpf autos}.

\subsubsection*{Acknowledgements} 
The first author would like to thank Michael L\"onne and S\"onke Rollenske for this nice opportunity of collaboration, and  grateful acknowledges two invitations to Marburg University. She is also grateful to Jongil Park for his guidance and his constant support, which allowed to be interested in signature of surface bundles. 
She was partially supported by BK21 PLUS SNU Mathematical Sciences Division.

The second author gratefully acknowledges the hospitality at Marburg university on two occasions and the support of the ERC 2013 Advanced Research Grant 340258-TADMICAMT. 
The third author acknowledges support of the DFG through the Emmy Noether program and partially through CRC 701. He is grateful to Jongil Park for the invitation to the 2015 SNU Topology Winter School, where the project was initiated, and to Fabrizio Catanese for their earlier collaboration on Kodaira fibrations.

Both the second and the third author thank Fabrizio Catanese for the invitation to the 2016 Workshop on Arithmetic and Geometry, where some of the results reached their final form. 

\section{Notation and formulas}

Our construction of double  Kodaira fibrations, as defined in Definition \ref{defin: double etale Kodaira}, starts from a product of curves together with the branch divisor $D$, so we introduce some notation for this.

\begin{defin}\label{defin: virtual Kodaira}
A virtual Kodaira fibration consists of the following data:
\begin{itemize}
 \item A product $F\times B$ of curves of genus at least $2$.
  \item A curve $D\subset F\times B$ such that both projections restrict to unramified coverings on $D$. 
 \item A surjection  $\theta\colon \pi_1(\hat F) \to G$ where $G$ is a finite group and $\hat F$ is  a general fibre of $F\times B \setminus D\to B$. 
 
 If $\gamma$ is a small loop in $\hat F$ around an intersection point of $D$ and $F$ then the order of $\theta(\gamma)$ in $G$ is called the local ramification order of $\theta$. We assume that the local ramification order is at least $2$ for each intersection point.
 
 We assume in addition the following compatibility condition for $\theta$: If $D_0$ is a component of $D$ and $\gamma$ and $\gamma'$ are small loops in $F$ around two points of $F\cap D_0$ then $\theta(\gamma)$ and $\theta(\gamma')$ are conjugate in $G$.
 
\end{itemize}
We call a virtual Kodaira fibration $\ka$ realisable, if $\theta$ is the restriction of a homomorphism $\Theta\colon \pi_1(F\times B \setminus D)\to G$.

We also extend Definition \ref{defin: double etale Kodaira} to virtual Kodaira fibrations in the obvious way.
\end{defin}

\begin{rem} 
The compatibility condition on $\theta$ is necessary for $\theta$ to be extendible and thus for $\ka$ to be realisable. 
 This follows from the fact that meridians of a connected divisor 
inside a smooth complex variety belong to a single
conjugacy class of the fundamental group of the complement.
So the claim is true for any extension of $\theta$ and hence we put it as a condition for
$\theta$ itself.
\end{rem}

Definition \ref{defin: virtual Kodaira} is motivated by the following:
\begin{prop}\label{prop: realisable}
 If $f\colon S\to F\times B$ is a double \'etale Kodaira fibration such that $f$ is a Galois cover  then the unramified cover $S\setminus \inverse f D\to F\times B\setminus D$ induces a realisable virtual Kodaira fibration $\ka(S) = (F\times B, D, \theta\colon \pi_1(\hat F) \to G)$.
 
 Conversely, for every realisable virtual Kodaira fibration $\ka = (F\times B, D, \theta)$ there exists a smooth $G$ cover $f\colon S\to F\times B$ which is a double \'etale Kodaira fibration.
\end{prop}
\begin{proof}
Let us explain how to construct a double \'etale Kodaira fibration from a realisable virtual Kodaira fibration; the converse is then clear.

Since $\ka$ is realisable, there is $\Theta\colon \pi_1(F\times B \setminus D)\to G$ extending $\theta$. Then the unramified cover $\hat S\to F\times B \setminus D$ corresponding to $\Theta$ can be compactified to a ramified Galois cover $f\colon S\to F\times B$ by the Riemann extension theorem \cite{GR58}.

The compact complex surface $S$ is smooth, because the branch locus $D$ is assumed to be smooth. We claim that the projection $f_1\colon S\to B$ is a holomorphic submersion but not isotrivial: for any point $b\in B$ the fibre $S_b$ is a branched cover of $F$, branched over $D_b = D\cap F\times \{b\}$ with monodromy given by $\theta$, where we use that the projection $D\to B$ is \'etale to guarantee that $D_b$ consists of the same number of disjoint points for every $b$. In particular, every fibre is smooth and $f_1$ is a submersion.

Assume $f_1$ is isotrivial, that is,  all fibres are abstractly isomorphic to a fixed curve $S_b$. Then, since there are only finitely many non-constant holomorphic maps $S_b\to F$ and a Galois cover of $F$ is determined by its branch locus and monodromy, we infer that there are only finitely many isomorphism classes of pairs $(F, D_b)_{b\in B}$. But this is impossible because $\Aut F$ is finite and the projection map $D\to F$ is surjective. 

Therefore the holomorphic submersion $f_1$ is not isotrivial and $S\to B$ is a Kodaira fibration. Repeating the argument for the projection to $F$ shows that $S\to F\times B$ is indeed a double \'etale Kodaira fibration.
\end{proof}

\begin{rem}
Let $D\subset F\times B$ be a curve mapping \'etale to both factors. 
By taking a Galois cover $g\colon \tilde B\to B$ dominating all components of $D$ one can arrange after pullback that $(\id\times g)^*D\subset F\times \tilde B$ is a union of graphs of \'etale maps from $\tilde B$ to $F$. 

It is an intriguing question if it is always possible to find a common covering   $g\colon \tilde B\to B$ and $h\colon \tilde B \to F$ such that $(h\times g)^*D \subset \tilde B\times \tilde B$ is a disjoint union of graphs of automorphisms. 

This lead to the notion of \emph{standard Kodaira fibration} in \cite{cat-roll09}.
\end{rem}

For simplicity of notation, we restrict to graph type in the next Definition; formulas for the general case can be found in \cite{cat-roll09}.
\begin{defin}
Let  $\ka= (F\times B, D = \sum D_i, \theta\colon \pi_1(\hat F) \to G)$ be a virtual Kodaira fibration of graph type, and let $r_i$ be the local ramification order at $D_i\cap F$.  We define the virtual Chern classes, signature and slope of $\ka$ to be
\begin{align*}
& c_2(\ka) = e(B) e(\tilde F) = |G|e(B)\left( e(F) - \sum_{i=1}^m \frac{r_i-1}{r_i}  \right), \\
 &c_1^2(\ka) = 2c_2(S) - |G| e(B) \sum_{i=1}^m \frac{r_i^2-1}{r_i^2},\\
 & \sigma(\ka) = \frac 1 3 \left( c_1^2(\ka)-2c_2(\ka)\right) = -\frac{ |G| e(B)}{3} \sum_{i=1}^m \frac{r_i^2-1}{r_i^2}, \\
 & \nu(\ka) = \frac{c_1^2(\ka)}{c_2(\ka)}.
\end{align*}

If $g\colon \tilde B \to B$ is a finite \'etale map, then the pullback of the virtual Kodaira fibration $\ka$ is $g^* \ka = (\tilde F\times B,(g\times \id_F)^*D, \theta)$. 
\end{defin}
Note that for a realisable Kodaira fibration these invariants coincide with the usual ones by \cite{cat-roll09}.

  By the tautological construction of \cite{cat-roll09} every virtual Kodaira fibration is realisable after some finite \'etale pullback. We will revisit this in Section \ref{sect: effective tautological}, keeping track exactly of the degree of the necessary pullback, which is needed to compute the signature. For convenience we introduce a notation.
\begin{defin}
 Let $\ka= (F\times B, D, \theta)$ be a virtual Kodaira fibration. Let $g\colon \tilde B \to B$ be a map of minimal degree such that $g^*\ka$ is realisable. We call $\tilde b =\tilde b(\ka) = g(\tilde B)$ the realisation-genus of $\ka$ and $\tilde \sigma  = \sigma(g^*\ka)$ the realisation-signature of $\ka$.
\end{defin}
Recall that by \cite{Meyer} 
we have $\tilde \sigma(\ka)\in 4\IN$ for a realisable virtual Kodaira fibration.

As suggested by the above, ramified coverings of algebraic curves play a prominent role in the construction of examples, and we introduce the following notation.
\begin{defin}\label{def: ram type}
Let $g\colon B\to B/G$ be a possibly ramified Galois cover of algebraic curves with Galois group $G$. 
The ramification type of $g$ is the tuple $(q\mid r_1,\cdots,r_m)$ where $q$ is  the genus of the quotient curve $B/G$, $m$ is the number of branch points and the  $r_i$ are the corresponding ramification multiplicities. 

If  $G$ is a cyclic group generated by an automorphism $\phi$ of $B$, then we call the corresponding $(m+1)$-tuple the ramification type of $\phi$.
\end{defin}

\section{Effective tautological construction}\label{sect: effective tautological}

In \cite{cat-roll09} 
the tautological construction was used to show that every virtual Kodaira fibration is realisable after finite \'etale pullback. Since the focus in loc.\ cit.\ was on the slope, which is invariant under pullback, computing the degree of the pullback map was not important. However, to control the signature, we are interested in the smallest possible pullback.

\subsection{Set-up of notations and idea of construction}
Let $S$ be a product of curves and $\pi\colon S \to B$ be the projection on one factor and $F$ a general fibre. Let $D\subset S$ be a divisor, such that $\pi|_D$ is unramified.

We let $\hat S = S\setminus D$ and $\hat F= F\setminus D$ and $\iota\colon \hat F \to \hat S$ the inclusion. $\hat \pi = \pi_{\hat S}\colon \hat S \to B$ is a differentiably locally trivial fibre bundle with fibre $\hat F$. Let $\theta\colon \pi_1(\hat F)\onto G$ be a surjective homomorphism to a finite group satisfying the compatibility condition of Definition \ref{defin: virtual Kodaira} and let  $\tilde F\to F$ be the corresponding ramified cover, branched over $F\cap D$. 

In order to construct an actual Kodaira fibration as a $G$-cover of $F\times B$ we want to extend the representation $\theta$ to a representation $\Theta\colon \pi_1(\hat S) \to G$ (see Proposition \ref{prop: realisable}). The idea is now, that one considers an appropriately chosen subspace $Z\subset \hat S$ such that $\pi_1(\hat F)$ and the image of $\pi_1(Z)$ generate the full fundamental group. Then $\Theta$ is uniquely determined by $\theta$ and $\theta' = \Theta|_{\pi_1(Z)}$ and, thus the existence of $\Theta$  is equivalent to the existence of $\theta'$ satisfying some compatibility relations with the given $\theta$.

\subsection{Tautological construction if $D$ contains a graph}
In this section we analyse the special case in which there is a component $D_0\subset D$ such that $D_0\to B$ has degree $1$; we say then that $D$ contains graphs. Note that $-D_0^2 = 2b-2$ where $b= g(B) = g(D_0)$. 

\begin{lem}[\cite{cat-roll09}] 
 Let $D_0$ be a graph in $D$, let $T$ be a tubular neighbourhood of $D_0$ in $S$ and $T_0 = T\setminus D_0$. Let $\gamma_0$ be a generator of the fundamental group of $F\cap T_0\simeq S^1$, that is, a small loop in the fibre $F$ around the puncture $D_0\cap F$.
  Then there is a presentation 
\begin{equation}\label{eq: presentation T_0}
 \pi_1(T_0) = \langle\alpha_1, \dots, \alpha_b, \beta_1, \dots, \beta_g, \gamma_0\mid \prod[\alpha_i, \beta_i]=\gamma_0^{2b-2}, \text{$\gamma_0$ central}\rangle
\end{equation}
and a diagram with exact rows
  \[ \begin{tikzcd} \{1\} \rar & \langle\gamma_0\rangle\dar \rar & \pi_1(T_0)\dar\rar & \pi_1(B)\dar[equal]\rar & \{1\}\\
  \{1\}\rar &  \pi_1(\hat F) \rar & \pi_1(\hat S)\rar & \pi_1(B)\rar &\{1\}
  \end{tikzcd}\]
  Consequently, $\pi_1(\hat S)$ is generated by the image of $\pi_1(T_0)$ and $\pi_1(\hat F)$.
\end{lem}
Note that the surjectivity fails if we do not assume $D_0$ to be isomorphic to $B$.

\begin{prop} \label{prop: tautological graph}
Assume that $D$ contains graphs.

Then $\theta$ is the restriction of a homomorphism $\Theta\colon \pi_1(\hat S) \to G$ if and only if there exists a graph $D_0\subset D$ and a  homomorphism $\theta'\colon \pi_1(T_0)\to G$, where $T_0$ is as above, with the following properties:
\begin{enumerate}
 \item $\theta(\gamma_0) = \theta'(\gamma_0)$
 \item For all $x\in \pi_1(\hat F)$ and all $y\in \pi_1(T_0)$ we have \[ \theta(x) = \theta'(\inverse y ) \theta(yx\inverse y) \theta'(y).\]
\end{enumerate}
If $G$ is abelian, then there exists $\theta'$ satisfying these conditions if and only if 
$\theta(\gamma_0^{2b-2}) = 0$ and  $\theta$ is invariant under the monodromy action of $\pi_1(B) $ on $\Hom(\pi_1(\hat F), G)$.
\end{prop}
\begin{proof}
We consider the following diagram of group homomorphisms:
\[ \begin{tikzcd}
     {}& \pi_1(\hat F)\arrow{dr}\arrow[out = 0, in = 135]{drrr}[swap]{\theta}\\
     \langle\gamma_0\rangle\arrow{ur}\arrow{dr} & & \pi_1(\hat S)\arrow[dashed]{rr}{\exists ? \Theta}&& G\\  
     & \pi_1(T_0)\arrow{ur}\arrow[dashed, out = 0, in = -135]{urrr}{\exists ? \theta'}
     \end{tikzcd}
\]
Note that since $\pi_1(\hat F)$ is a normal subgroup every element of $a\in \pi_1(\hat S)$ can be written as $a=xy=(x\gamma_0^k)(\gamma_0^{-k}y)$ with $x\in \pi_1(\hat F)$ and $y\in \pi_1(T_0)$, which are uniquely determined up to multiplication with $\gamma_0$ as indicated.

If $\Theta$ exists then we can define $\theta'$ by composition and  have 
\begin{equation}\label{eq: homomorphism}
\Theta(a) = \Theta(x)\Theta(y) = \theta(x)\theta'(y). 
\end{equation}

Conversely, if we are given $\theta'\colon \pi_1(T_0)\to G$ then \eqref{eq: homomorphism} defines a well-defined map of sets if and only if \refenum{i} holds. It is also straightforward to check, that \eqref{eq: homomorphism} defines a homomorphism if and only if  in addition \refenum{ii} holds.

If $G$ is abelian then $\theta'$ factors over $\pi_1(T_0)_{\text{ab}}=H_1(B)\oplus \langle\gamma_0\mid \gamma_0^{2b-2}\rangle$, so there exists a homomorphism $\theta'$ satisfying \refenum{i} if and only if $\theta(\gamma_0)^{2b-2}=0$; the second condition translates exactly into $\theta$ being invariant under the monodromy action. This concludes the proof.
\end{proof}

\subsection{Tautological construction: general case}
If $D$ does not contain a graph, then no tubular neighbourhood of a component of $D$ contains enough information to reconstruct $\Theta$, so instead we pull back the fibration to a wedge of circles.
\begin{lem}
 Let $b$ be the genus of the curve $B$ and choose disjoint based loops $\alpha_1, \beta_1, \dots, \alpha_b, \beta_b$ in $B$ representing a standard set of generators for $\pi_1(B)$. Let $\iota\colon \bigvee_{j=1}^{2b} S^1 \into B$ be the inclusion of the wedge of the chosen loops and $\hat\iota \colon \iota^*\hat S \to \hat S$ be the induced inclusion.
 Then there is a commutative diagram with exact rows and columns:
 \begin{equation}\label{diag: circle wedge}
\begin{tikzcd}
 {} &  & 1\dar & 1\dar\\
  {} &  & \pi_1(\hat F) \rar[equal]\dar & \pi_1(\hat F) \dar& \\
 1\rar& \ker(\hat\iota_*)  \rar\dar{\isom} & \pi_1(\iota^*\hat S)\rar\dar{p_*}& \pi_1(\hat S)\rar\dar{p_*}& 1\\
1\rar& \langle\langle \prod_i[\alpha_i, \beta_i]\rangle\rangle \rar & \langle \alpha_i, \beta_i\rangle \rar\dar&\pi_1(B)\rar\dar& 1\\
{}& &1&1\\
\end{tikzcd}
\end{equation}
\end{lem}
\begin{proof}
 The last row is the usual presentation of the fundamental group of a curve, and last two columns are the exact sequences associated to the fibre bundle $\hat S\to B$ and its pull-back.
 
 The remaining claims follow from an easy diagram chase.
\end{proof}

\begin{prop}\label{prop: tautological general}
 In the above situation the following holds:
 \begin{enumerate}
  \item The homomorphism $\theta$ is the restriction of a homomorphism  $\tilde \Theta\colon \pi_1(\iota^*\hat S)\to G$
if and only if there exist loops $\tilde \alpha_i, \tilde\beta_i\in \pi_1(\iota^*\hat{S})$ with  $p_*\tilde\alpha_i = \alpha_i$ and $p_*\tilde \beta_i = \beta_i$
and a homomorphism $\theta'\colon \langle \tilde \alpha_i, \tilde \beta_i\rangle\to G$ satisfying
  \begin{equation}\label{eq: monodromy invariant general}
\theta(x)=\theta'(y^{-1})\theta(yxy^{-1})\theta'(y) \text{~for all~} x\in\pi_1(\hat{F}) \text{~and for all~} y\in \langle\tilde{\alpha_i}, \tilde{\beta_i}\rangle_{i=1}^{b} 
 \end{equation}
\item
If the extension $\tilde \Theta $ exists then it descends to a homomorphism $\Theta\colon \pi_1(\hat S)\to G$ if and only if $\tilde \Theta$ is trivial on $\ker \hat{\iota}_*$.
 \end{enumerate}
\end{prop}
\begin{proof}
 Note that the quotient of $\pi_1(\iota^*\hat S)$ by $\pi_1(\hat F)$ is a free group, so any choice of lift of the generators splits the middle column of \eqref{diag: circle wedge}. Hence any choice of homomorphism $\pi_1(\bigvee_j S^1)\to G$ lifts to a map of sets $\pi_1(\iota^*\hat S)\to G$. From the definition of the multiplication in the semi-direct product we read of that the second condition is equivalent to this map of sets being a group homomorphism. 
 
 This group-homomorphism descends to a homomorphism $\Theta\colon \pi_1(\hat S) \to G$ if and only if it is trivial on the kernel.
\end{proof}

Before we state the specialisation of this result for $G$ abelian we make some preliminary considerations: write $D= \sum_{i=1}^m D_i$ as a sum of components and fix for each $i$ a small loop in $F$ around one of the points in $F\cap D_i$. Assume that $G$ is an abelian group, written additively.  Then we define the \emph{global extension obstruction}  of $\theta$ as
\[  o(\theta) = \sum_i \left(\deg\left( D_i \to F\right)\right)\theta(\gamma_i).\]
Note that since $\theta$ satisfies the compatibility condition of Definition \ref{defin: virtual Kodaira} and $G$ is abelian the element $o(\theta)\in G$ does not depend on the choices made.
\begin{cor}\label{cor: tautological general abelian}
Assume that in the situation of Proposition \ref{prop: tautological general} $G$ is abelian. Then $\theta$ is the restriction of a homomorphism  $\tilde \Theta\colon \pi_1(\iota^*\hat S)\to G$ if and only if $\theta$ is invariant under the monodromy action of $\pi_1(B) $ on $\Hom(\pi_1(\hat F), G)$.

There exists such an extension $\tilde \Theta$ which is the restriction of a homomorphism $\Theta\colon \pi_1(\hat S) \to G$ if and only if in addition the global extension obstruction  $o(\theta) = 0$ in $G$.
\end{cor}
It is not hard to deduce the condition for $G$ abelian stated in Proposition \ref{prop: tautological graph} from this result.
\begin{proof}
The conjugation action of two different lifts say $\alpha_i$ differs by an inner automorphism of $\pi_1(\hat F)$. Thus all lifts act in the same way on $\theta\in \Hom(\pi_1(\hat F), G)$, because $G$ is abelian and \eqref{eq: monodromy invariant general} just means that $\theta$ is monodromy invariant.

To conclude we have to show that, modulo the commutator,  the kernel of $\hat \iota_*$ is generated by the obstruction element $o(\theta)$. To see this, we choose in $F\times B$ a 
horizontal section,  $B_0$ isomorphic to $B$ and a small tubular neighbourhood $T$ of $D$ and consider the subspace $Z=B_0\cup T\cup F_0\setminus D$ 
where $F_0$ is the fibre over the base-point in $B$. Denote the lifts of the generators of $\pi_1(B)$ given by $B_0\cap \iota^*\hat S$ by $\tilde \alpha_i, \tilde \beta_i$ and let $\delta_j$ be horizontal loops around the punctures $B_0\cap D$. 

Note that in $\pi_1(B_0\setminus D)$, and thus in $\pi_1(Z)$ we have (modulo the commutator)  the relation $\prod_j\delta_j\prod_i[\tilde \alpha_i, \tilde\beta_i]=1$.
If the loop $\delta_j$ runs around a puncture in $B_0\cap D_i$ then $\delta_j$ is conjugate to $\gamma_i$, the  loop around $D_i$ contained in $F$ we chose above.
Thus modulo the commutator we have in $\pi_1(Z)$, and therefore in $\pi_1(\hat S)$ the relations
\[ \prod_j\delta_j = \prod_i\gamma_i^{\deg\left( D_i \to F\right)} \text{ and } \prod_i\gamma_i^{\deg\left( D_i \to F\right)} \prod_i[\tilde \alpha_i, \tilde\beta_i]=1.\]
Since $\ker \iota_*$ is normally generated by  $p_*\left(\prod_i\gamma_i^{\deg\left( D_i \to F\right)} \prod_i[\tilde \alpha_i, \tilde\beta_i]\right)$, the global extension obstruction $o(\theta)$ normally generates $\ker \hat \iota_*$ by \eqref{diag: circle wedge}.

Thus an extension $\tilde\Theta$ is the restriction of $\Theta \colon \pi_1(\hat S)\to G$ if and only if 
\begin{align*}
 0&=\tilde\Theta\left(\prod_i\gamma_i^{\deg\left( D_i \to F\right)} \prod_i[\tilde \alpha_i, \tilde\beta_i]\right) \\
 &= \tilde\Theta\left(\prod_i\gamma_i^{\deg\left( D_i \to F\right)}\right)\\
 &=\theta\left(\prod_i\gamma_i^{\deg\left( D_i \to F\right)}\right)\\
 & =o(\theta),
 \end{align*}
where the commutator maps to $0$ in $G$ because $G$ is abelian.
\end{proof}
\subsection{Minimal pullbacks}
\begin{cor}\label{cor: minimal pullback degree}
Let $\ka = (F\times B, D, \theta\colon \pi_1(\hat F)\to G)$ be as above with $G$ an abelian group. Then there exists an \'etale cover $g\colon \tilde B\to B$ such that the  the pullback $g^*\ka$ is realisable. The minimal degree of such a $g$ is the least common multiple of the order of $o(\theta)$ in $G$ and $[\pi_1(B):\Stab_\theta]$, where $\Stab_\theta$ is the stabiliser of $\theta$ under the action of $\pi_1(B)$ on $\Hom(\pi_1(\hat F), G)$.
\end{cor}
\begin{proof}
 By Corollary \ref{cor: tautological general abelian} we need that the monodromy action fixes $\theta$ and that the global extension obstruction  vanishes for the $g^*\ka$. 
 Let $H$ be the subgroup of $\pi_1(B)$ corresponding to $g$. Then the first condition is satisfied if and only if $H\subset \Stab_\theta$ and the second condition is satisfied if and only if $\deg (g)\cdot o(\theta) = 0$ in $G$.  Since the fundamental group of a curve of positive genus has quotients of every finite order one can always find a subgroup of $\Stab_\theta$ that has exactly the required index. 
\end{proof}

\begin{rem}[ Non-Galois extensions]\label{rem: non-Galois realisation}
Note that realisability of a virtual Kodaira fibration entails by definition that the ramified cover $S \to F\times B$ is itself Galois with Galois group $G$. Alternatively, one could impose a weaker condition and look for an arbitrary  ramified cover $S\to F\times B$ such that the restriction to the fibre is the given Galois cover. Let us analyse quickly how this affects the minimal realisation degree. 

Assume $\ka = (F\times B, D, \theta\colon \pi_1(\hat F) \to G) $ is a virtual Kodaira fibration and that $g\colon \tilde B\to B$ is the minimal pullback such that $g^*\ka$ is realisable corresponding to a double Kodaira fibration $f\colon\tilde S\to F\times \tilde B$. Then possibly there is a finite group $H$ acting freely on $\tilde{S}$ and $\tilde B$ such that $\tilde{S}\to \tilde B$ is $H$-equivariant. If we divide out this action we get a diagram 
\[
\begin{tikzcd}
\tilde F\rar[hookrightarrow]\dar &  \tilde S \rar \dar{\text{\scriptsize Galois}}& \tilde S/H \dar{\text{\scriptsize non-Galois}}\\
F\rar[hookrightarrow]&  F\times \tilde B\dar \rar &  F\times \tilde B/H\dar \rar & F\times B\dar\\
 &\tilde B\rar & \tilde B/H\rar & B
\end{tikzcd},
\]
and thus a Kodaira fibration with smaller signature. 

We leave the exploration of this phenomenon for future research and only give an example of a non-Galois cover of a fibration that on each fibre restricts to a Galois cover.

Let $N\vartriangleleft S_3$ be the cyclic normal subgroup of order $3$ and $\IZ/2\isom A<S_3$ be one of the non-normal subgroups of order $2$. 
Consider a curve $\tilde F_7$ of genus $7$ with a free $S_3$-action and 
quotient $F_2 = F_7/S_3 $ of genus $2$ and quotient $F_3= F_7/N$ of genus $3$.
Let $B_3$ be a curve of genus $3$ with a quotient $B_2= B_3/A$ 
by a free automorphism of order 2.

Then we get the following diagram
\[
\begin{tikzcd}
 F_7\rar[hookrightarrow]\dar{/N} &   F_7\times B_3  \rar \dar{/N}[swap]{f}
& ( F_7\times B_3)/A \dar{\text{\scriptsize non-Galois}}[swap]{\bar f}\\
F_3\rar[hookrightarrow]&  F_3\times B_3\dar \rar & (F_3\times B_3)/A\dar\rar & F_2\times B_2\dar \\
 & B_3\rar & B_2\rar & B_2
\end{tikzcd},
\]
where the horizontal maps on the left are the inclusions of the fibres
over the base $B_3$ and the horizontal maps in the middle are the quotient
maps associated to actions of $A$, the free action on $B_3$ for the bottom
row, the diagonal action for the other rows.
Then the covering $\bar f$ in the third column is not Galois but it extends the Galois cover $F_7 \to F_3$ on every fibre. The point is that the fibre-wise Deck transformation does not glue to a global Deck transformation because it is not invariant under the monodromy action in the fibre bundle.
\end{rem}

\section{Explicit computation of monodromies}\label{sect: monodromy computation}
Let $D\subset F\times B$ a divisor in a product of curves. Considering the projections 
 \[\begin{tikzcd} F\times B \rar{q} \dar[swap]{p}& F \\ B\end{tikzcd}.\]
we assume $D\to B$ to be an unramified covering.
We fix a general point $b_0\in B$ and identify $F$ with the fibre of $p$ over $b_0$, so that we can write $F\cap D  = F\times \{b_0\} \cap D$ and $\hat F = F \setminus F\cap D$.

Corollary \ref{cor: minimal pullback degree} suggests we should compute the  monodromy action of $\pi_1(B)$ on homomorphisms $\theta\colon \pi_1(\hat F)\to G$. We only address this for $G$ a  (finite) abelian group, when $\theta$ factors uniquely over $H_1(\hat F;\IZ)$.

We start by computing the monodromy action on relative homology, for which we introduce the following construction: for a loop $\alpha$ in $B$ based in $b_0$ and $x \in D\cap F$ let $\tilde \alpha^x$ be the unique lift of $\alpha$ to $D$ starting in $x$, and consider the bilinear pairing defined by 
\begin{equation}\label{eq: weighted transfer}
H_1(B;\IZ)\times H_0(F\cap D; G) \to H_1(F, F\cap D; G), \qquad (\alpha, x) \mapsto \hat \alpha(x) := q_*\tilde \alpha^x, 
\end{equation}
where we identify $x\in F\cap D$ with its homology class.
We call this the weighted transfer pairing since (for $\IZ$-coefficients) $\hat\alpha(F\cap D) = q_*p^!\alpha$ where $p^!\colon H_1(B) \to H_1(B, b_0)\to H_1(D, D\cap F)$ is the transfer map.

\begin{thm}\label{thm: monodromy action}
 Let $G$ be an abelian group and $D\subset F\times B$ a divisor in a product of curves such that its projection to $B$ is \'etale and let 
 \[\chi(-)_*\colon \pi_1(B) \to \Aut(H_1(F, F\cap D;G))\]
 be the monodromy representation. Then the  action of $\alpha \in \pi_1(B)$ on an element $\theta\in H_1(F, F\cap D; G)$ is given by the weighted transfer pairing of the homology class of $\alpha$ with the boundary of $\theta$ as follows:
 \[\chi(\alpha)_*\theta = \theta +  \hat \alpha(\del\theta).\]
\end{thm}
\begin{proof}
Pull back the configuration $D\subset F\times B$ to the interval $I$ via the path $\alpha$. Then the pullback of $D$ is the union of the different liftings of $\alpha$ to $D$ and defines a braid in the surface $F$ as depicted in Figure \ref{fig: surface braid}.
\begin{figure}[h]\caption{Surface braid induced by $\alpha\in \pi_1(B)$}\label{fig: surface braid}
\begin{center}
\begin{tikzpicture}[
fibre/.style = {}, 
horizontal/.style = {thin, dashed},
lift1/.style = {red, thick},
lift2/.style = {blue, thick},
cycle/.style = {thick},
points/.style = {black}
]
\coordinate (x0) at (0,-.5);
\coordinate (y0) at (0,.-1.5);
\coordinate (x1) at (6,-.5);
\coordinate (y1) at (6,.-1.5);
\coordinate (x2) at (10,-.5);
\coordinate (y2) at (10,.-1.5);

\begin{scope}[fibre]
\draw (0,0) ellipse [x radius=1cm,y radius=2cm];
\draw[bend right] (0,0) to ++(0,1) to ++(0,-1);
\draw (6,0) ellipse [x radius=1cm,y radius=2cm];
\draw[bend right] (6,0) to ++(0,1) to ++(0,-1);
\draw (10,0) ellipse [x radius=1cm,y radius=2cm];
\draw[bend right] (10,0) to ++(0,1) to ++(0,-1);
\end{scope}
\foreach \c in { 0, 1} {\draw[white,line  width = .2cm] (0.5,\c) -- (5.5, \c);};
\begin{scope}[horizontal]
\foreach \c in {-2, 0, 1, 2} {\draw (0,\c) -- (6, \c);};
\end{scope}

\begin{scope}[lift1]
 \draw[white, line width = .2cm] (x0) to[out=0, in = -120] (1.5, 0);
 \draw[white, line width = .2cm]  (2.4,1.1) to [out=30, in = 180] (y1);
 \draw[postaction = {mid arrow}] (x0) to[out=0, in = -120] (1.5, 0);
\draw[dashed] (1.5, 0) to[out =60, in =-150] (2.5,1);
\draw[postaction={decorate,decoration={markings,mark=at position .4 with {\arrow{stealth}}}}] (2.5, 1) to [out=30, in = 180] (y1);

\node at (3,1.5) {$\tilde\alpha^x$};

\draw[postaction={decorate,decoration={markings,mark=at position 0.42 with {\arrow{stealth}}}}] (x2) to [out = 120, in = 180] (10, 1.5)  to[out = 0, in =70]    (y2);
\node at (11.5,1) {$q_*\tilde\alpha^x$};

\end{scope}

\begin{scope}[lift2]
\draw[white, line width =.2cm] (y0) to[out = 0, in = 180] (x1);
\draw[postaction = {mid arrow}] (y0) to[out = 0, in = 180] node[below] {$\tilde \alpha^y$} (x1);

\draw[postaction = {mid arrow}] (y2) to[out = 10, in = -10, looseness = 2] (x2);
\node at (11.5,-1) {$q_*\tilde\alpha^y$}; 
\end{scope}

\begin{scope}[cycle]
\draw[bend left, postaction = {mid arrow}] (y0) to node[left] {$\theta$} (x0);
\draw[bend left, postaction = {mid arrow}] (y2) to node[left] {$\theta$} (x2); 
\end{scope}

\draw[->, thick] (7.5,0) -- ++(1,0);
\foreach \P in { x0,y0,x1,y1, x2, y2} {\fill[points] (\P) circle (2pt);};
\path (x0) node [left] {$x$} (y0) node[left] {$y$};

\node at (3, -2.5) {$F\times I$};
\node at (10, -2.5) {$F$};
\end{tikzpicture}
\end{center}
 \end{figure}
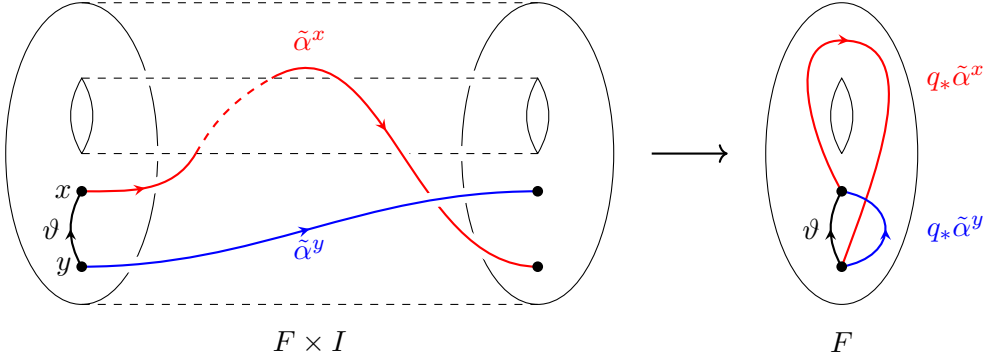
We only need to determine the action on generators of $H_1(F, F\cap D;G)$ that are not in $H_1(F;G)$, so it is sufficient to compute the action of $\alpha$ on the class of a path $\theta$ connecting two points in $F\cap D$. Write $\del \theta = x-y$. Then the composition of $\inverse{({\tilde\alpha}^y)}$, $\theta$, and $\tilde\alpha^x$, considered as paths in $F\times I$, is homotopy equivalent to $\chi(\alpha)_*\theta$,  relative to the endpoints. Since the projection to $F$ is a homotopy equivalence we get the claimed formula.
\end{proof}

\begin{rem}
In Theorem \ref{thm: monodromy action}, we have seen the monodromy action of $\alpha\in\pi_1(B)$ on an element $\theta\in H_1(F,F\cap D)$. In this remark, we'll explain the monodromy action in more detail from the point of view of braid groups (see \cite{birman2016braids}).

Let $D\subset B\times F$ be a smooth divisor which is unramified over $B$. Then the pair $(B\times F,D)\xrightarrow{p}B$ is a locally trivial fibration with the fibre $(F,F\cap D)$, the surface $F$ with $d$ distinguished points $\{x_0,\cdots,x_{d-1}\}$. 
Its monodromy homomorphism $\chi\colon\pi_1(B)\to \mathrm{Mod}(F,F\cap D)$, which maps to the marked mapping class group, takes values in the subgroup $\mathrm{Br}_d(F)$, the $d$-stranded surface braid group, because $B\times F\to B$ is a trivial bundle. This follows from the generalised Birman exact sequence 
\[1\to \mathrm{Br}_d(F)\xrightarrow{\text{push}} \mathrm{Mod}(F,d)\xrightarrow{\text{forget}} \mathrm{Mod}(F)\to 1.\] Therefore, the monodromy action of the marked mapping class $\chi(\alpha)\in \mathrm{Mod}(F,F\cap D)$ on the relative homology group $H_1(F,F\cap D)$ is determined by the action of the image under the push map of the surface braid $\beta(\alpha)\in \mathrm{Br}_d(F)$, where $\beta\colon\pi_1(B)\to\pi_1(C(F,d))=\mathrm{Br}_d(F)$ is given on representatives by  $\alpha(t)\mapsto q_*(p^{-1}(\alpha(t))\cap D)$.
On the other hand, we have an exact sequence 
\[1\to H_1(F)\to H_1(F,F\cap D)\xrightarrow{\partial} H_0(F\cap D)\to H_0(F)\to 1.\] 
Hence $H_1(F,F\cap D)$ is isomorphic to a direct sum of a natural subgroup $H_1(F)$ and a (non canonical) complement $H$ isomorphic to $\im\partial$. We can take $H$ as the subgroup generated by the paths $\delta_i$ in a fixed disk connecting two marked points $x_0$ and $x_i$. 
With respect to a basis subordinate to this decomposition, the induced action on the relative homology is given by
\[\chi(\alpha)_*=\left[\begin{array}{c|c}
\id & \psi(\alpha)\\
\hline
0 & \sigma(\alpha)
\end{array}\right]\]
since the action is trivial on $H_1(F)$.
In order to look at the action on $H$, we write the surface braid $\beta(\alpha)$ as a product of a pure surface braid $\bar\beta(\alpha)$ and a braid $\pi_{\beta}(\alpha)$ supported on the disc used above for the definition of $H$. 
The braid $\pi_\beta(\alpha)$ contributes as a matrix $\sigma(\alpha)$ easily obtained from
the permutation action of the braid on the points $x_i$
 and $\bar\beta(\alpha)$ contributes as a matrix $\psi(\alpha)$ where the $i$-th column gives the difference in homology $H_1(F)$ of the loops traced by $x_i$ and $x_0$ on the surface $F$.

To verify this, 
first write $\bar{\beta}(\alpha)=(\bar{\beta}_0(\alpha),\bar{\beta}_1(\alpha),\cdots,\bar{\beta}_{d-1}(\alpha))$ so that the $i$-th strand $\bar{\beta}_i(\alpha)$ is the trace of $x_i$.
Then the push map along each strand of the pure braid acts trivially on the subgroup $L$ of $H_1(\hat{F})$ generated by $\gamma_i$'s, and
for each $\beta\in H_1(F)\subset H_1(\hat{F})$,
\[\mathrm{Push}(\bar{\beta}_i(\alpha))_*(\beta)=\beta+\langle\bar{\beta}_i(\alpha),\beta\rangle\gamma_i.\]
Hence, for any $\theta\in H_1(F,F\cap D;G)$, \[\mathrm{Push}(\bar{\beta}_i(\alpha))_*\theta=\theta+[(\bar{\beta}_i(\alpha)]\otimes\theta(\gamma_i)\] by Poincare duality. Therefore, by \cite{goldberg74},
\[\mathrm{Push}(\bar{\beta}(\alpha))_*\theta-\theta=\sum_{i=0}^{d-1}[\bar{\beta}_i(\alpha)]\otimes\theta(\gamma_i)=\sum_{i=1}^{d-1}=([\bar{\beta}_i(\alpha)]-[\bar{\beta}_0(\alpha)])\otimes\theta(\gamma_i).\]

\end{rem}

Let us set up the notations for the rest of the section. For later use it is convenient to decompose $D = \bigsqcup_i D_i$ into connected components 
 and consider the \'etale maps $p_i = p|_{D_i}\colon D_i \to B$.  Then we write $F\cap D_i = \{ x_{ij}\}_j$ so that $D\cap F = \{x_{ij}\}_{i,j}$. We choose sufficiently small loops $\gamma_{ij}$ around $x_{ij}$,  positively oriented. For convenience choose also a symplectic basis $\alpha_1, \beta_1, \dots, \alpha_f,  \beta_f$ for $H_1(F, \IZ)$.
 With these choices we have $H_1(\hat F;\IZ) = \langle \alpha_k, \beta_k, \gamma_{ij}\rangle\slash \langle \sum \gamma_{ij}\rangle$. 
 
Now let $\theta \colon \pi_1(\hat F) \to  H_1(\hat F; \IZ)\to G$ be a homomorphism 
 to a (finite) abelian group satisfying the conditions of Definition \ref{defin: virtual Kodaira}, that is, $\theta(\gamma_{ij}) = g_i$ is independent of $j$. 

 By relative Alexander duality \cite[p.234, Step~3]{GH} capping with the orientation class gives a commutative diagram
 \begin{equation}\label{eq: Alexander duality}
  \begin{tikzcd}
   H^0(F\cap D; \IZ) \rar{\delta} \dar{\isom} & H^1(F, F\cap D;\IZ) \rar\dar{\isom}[swap]{\cap \zeta_F}  & H^1(F; \IZ)\dar{\isom}\\
   H_2(F, \hat F;\IZ) \rar & H_1(\hat F;\IZ) \rar & H_1(F;\IZ)
  \end{tikzcd},
 \end{equation}
and denoting by $\check x_{ij}$ the dual of $x_{ij}$, considered as a class in $H^0(F\cap D;\IZ)$, we have $\delta(\check x_{ij})\cap \zeta_F= \gamma_{ij}$.

Note that the universal coefficient theorem provided an identification of the cohomology group $H^1(F, F\cap D;\IZ) $ with $ \Hom (H_1(F, F\cap D; \IZ), \IZ)$ and thus
\[ \Hom (H^1(F, F\cap D;\IZ), G) \isom H_1(F, F\cap D; G) \isom H_1(F, F\cap D;\IZ)\tensor G.\]
Combining this with \eqref{eq: Alexander duality} we get a dual commutative diagram 
 \begin{equation}\label{eq: dual Alexander duality}
  \begin{tikzcd}
   H_1(F;G ) \rar \dar{\isom} & H_1(F, F\cap D;G) \rar{\del}\dar{\isom}[swap]{}  & H_0(F\cap D; G)\dar{\isom}\\
   H^1(F;G) \rar & \Hom (H_1(\hat F;\IZ), G) \rar & H^2(F, \hat F;G)
  \end{tikzcd}.
 \end{equation}
Following the isomorphisms gives immediately
\begin{lem}\label{lem: del theta}
 Considering $\theta\colon \pi_1(F)\to G $ as an element of $H_1(F, F\cap D; G)$ we have
 \[\del \theta  = \sum_{ij} x_{ij}\tensor\theta(\gamma_{ij}) = \sum_i \left(\textstyle\sum_j x_{ij}\right) \tensor g_i.\] 
\end{lem}

We can now make the computation of the stabiliser explicit.
\begin{cor}\label{cor: monodromy action graph type}
In the above notation the stabiliser of $\theta$ under the monodromy action is the kernel of the homomorphism $\iota\colon \pi_1(B) \to H_1(F;G)\subset  H_1(F, F\cap D;G)$, defined by 
\[\iota(\alpha) = \hat \alpha\left( \sum_{ij} x_{ij}\tensor\theta(\gamma_{ij})\right) = q_*\sum_i p_i^!\alpha \tensor g_i,\]
and consequently the index $[\pi_1(B)\colon \Stab_{\theta}]$ equals the cardinality of the groups $\im \iota\isom H_1(B;\mathbb{Z})/ker(\iota)$.

In particular, if $D = \bigsqcup_i \Gamma_{\phi_i}$ is a disjoint union of graphs then 
$\iota(\alpha) = \sum_i {\phi_i}_*\alpha \tensor g_i$.
\end{cor}
\begin{proof} As above we consider $\theta$ as an element of $H_1(F, F\cap D;G)$. 
 By Theorem \ref{thm: monodromy action} $\alpha \in \Stab_\theta$ if and only if $\theta = \chi(\alpha)_*\theta = \theta+\hat\alpha(\del \theta)$, that is, if and only if $\iota(\alpha) =0$, where we have used the description of $\del\theta$ from Lemma \ref{lem: del theta}. If follows from the definition \eqref{eq: weighted transfer} that $\iota$ is a homomorphism. 
 
 It remains to show that the image of $\iota$ is contained in $H_1(F;G)$, or equivalently from \eqref{eq: dual Alexander duality} that for all $\alpha$ \[0=\del\circ \iota(\alpha)= \del\left(\chi(\alpha)_*\theta-\theta\right) = \chi(\alpha_*)(\del \theta) - \del\theta.\]
 But the monodromy action on $\del \theta$ is indeed trivial, because it only permutes the intersection points of $D_i\cap F$ and $\theta(\gamma_{ij})$ does not depend on $j$ by assumption. 
\end{proof}

\section{Virtual Kodaira fibrations with small virtual signature}\label{sect: virtual with small signature}
In this section we delve into the intricacies of constructing and classifying virtual Kodaira fibrations of small virtual signature. The upshot is that while giving some numerical restrictions and constructing some examples is easy, working out a complete list turns out to be a larger endeavour.

\subsection{Numerical restrictions for graph type}
Since the signature of a realisable virtual Kodaira fibration is always divisible by $4$ we restrict to this case in the numerical classification. 
\begin{prop}\label{prop: numerical graph type}
  Let $\ka= (F\times B, D, \theta\colon \pi_1(\hat F)\to G)$ be a virtual Kodaira fibration of graph type and let $m=D.F$ be the number of punctures. If the virtual signature $\sigma(\ka)$ is at most $16$ and divisible by $4$ then we are in one of the  cases listed in Table \ref{tab: virt Kodaira graph}.

  \begin{table}\caption{Virtual Kodaira fibrations of graph type with small signature}\label{tab: virt Kodaira graph}
  \begin{tabular}{rrrrc}
  \toprule
  $\sigma(\ka)$ & $g(B)$ & $|G|$ & $r =(r_1, \dots, r_m)$ & group $G$\\
  \midrule



 
$4$ & $2$ & $8$ & $( 2 )$ & non-abelian \\    
 $4$ & $3$ & $2$ & $( 2, 2 )$ \\   
$4$ & $2$ & $4$ & $( 2, 2 )$ \\   
\midrule
 $8$ & $3$ & $8$ & $( 2 )$ & non-abelian\\   
 $8$ & $2$ & $16$ & $( 2 )$ & non-abelian\\ 
 $8$ & $5$ & $2$ & $( 2, 2 )$ \\   
 $8$ & $3$ & $4$ & $( 2, 2 )$ \\   
 $8$ & $2$ & $8$ & $( 2, 2 )$ \\   
 $8$ & $3$ & $2$ & $( 2, 2, 2, 2 )$ \\   
\midrule
 $12$ & $4$ & $8$ & $( 2 )$ & non-abelian \\   
 $12$ & $3$ & $12$ & $( 2 )$ & non-abelian\\   
$12$ & $2$ & $24$ & $( 2 )$ & non-abelian\\
 $12$ & $7$ & $2$ & $( 2, 2 )$ \\   
 $12$ & $4$ & $4$ & $( 2, 2 )$ \\  
 $12$ & $3$ & $6$ & $( 2, 2 )$ \\   
 $12$ & $2$ & $12$ & $( 2, 2 )$ \\   
 $12$ & $3$ & $4$ & $( 2, 2, 2 )$ &$\IZ/2\times \IZ/2$ \\   
 $12$ & $2$ & $8$ & $( 2, 2, 2 )$ &non-cyclic\\   
$12$ & $4$ & $2$ & $( 2, 2, 2, 2 )$ \\   
$12$ & $3$ & $2$ & $( 2, 2, 2, 2, 2, 2 )$ \\   
\midrule
$16$ & $5$ & $8$ & $( 2 )$ & non-abelian \\
$16$ & $3$ & $16$ & $( 2 )$ & non-abelian \\    
$16$ & $2$ & $32$ & $( 2 )$ & non-abelian \\
$16$ & $2$ & $27$ & $( 3 )$ & non-abelian \\   
 $16$ & $9$ & $2$ & $( 2, 2 )$ \\   
 $16$ & $5$ & $4$ & $( 2, 2 )$ \\   
 $16$ & $3$ & $8$ & $( 2, 2 )$ \\   
 $16$ & $2$ & $16$ & $( 2, 2 )$ \\   
 $16$ & $4$ & $3$ & $( 3, 3, 3 )$ \\   
$16$ & $2$ & $9$ & $( 3, 3, 3 )$ \\   
 $16$ & $5$ & $2$ & $( 2, 2, 2, 2 )$ \\   
  $16$ & $3$ & $4$ & $( 2, 2, 2, 2 )$ \\   
\bottomrule
 \end{tabular} 
  \end{table}
\end{prop}
\begin{proof}Let $d = |G|$.
In the inequality 
\[16\geq  \sigma(\ka)= \frac 23 (b -1){ d } \sum_{i=1}^m \frac{r_i^2-1}{r_i^2}\]
we know that $b\geq2$, $d\geq2$, $m\geq 1$ and $r_i\geq2$, thus $\frac{r_i^2-1}{r_i^2}\geq \frac 34$.
Hence, $32\geq (b-1)dm$ which together with the condition that the signature is divisible by $4$ leaves a small number of cases to consider.

We now proceed to exclude several cases until we arrive at the list given in the proposition; as in the table we will identify the cases by the tuple $(\sigma, g(B), |G|, r)$

First of all note that if $m=1$ then $\hat F$ is a curve of genus at least $2$ with only one puncture. So the loop around the puncture is a commutator in $\pi_1(\hat F)$ which is sent to an element of order $r_1$ in $G$. Thus, $G$ has to be non-abelian which excludes groups of order $2,3,4,5,7,9, 25$. The case  $(12, 5, 6, ( 2))$ is excluded, because $S_3$ does not contain an element of order $2$ in the commutator.

If $g(B)=2$ then all graphs are necessarily graphs of automorphisms and it is known from \cite{cat-roll09} that there can be at most three non-intersecting such graphs, so $m\leq3$ in this case.
In a similar fashion, if $g(B)=3$ then there cannot be more than six graphs, which excludes another case.

Lastly, in $\IZ/2^k$ no odd number of order two elements can sum up to zero which excludes the case 
$(12, 5, 2, ( 2, 2, 2 ))$ and restricts the possible groups in two other cases.
\end{proof}
\begin{rem}
 In some cases in the above proposition it is clear that the cover of $F$ factors through an \'etale cover, so it is tempting to do the unramified cover first and consider the pullback divisor in another product of curves. However, this will usually destroy the graph property.
\end{rem}

\subsection{Numerical classification of virtual double \'etale Kodaira fibrations of virtual signature $4$}
To classify the numerical possibilities in this case we need some further notation. Let $\ka= (F\times B, D, \theta\colon \pi_1(\hat F)\to G)$ be a double-\'etale virtual Kodaira fibration. For each component $D_i$ of $D$ let $d_i:=\deg (D_i \to B)$ and $e_i:= \deg (D_i \to F)$ be the degrees of the projections and let $r_i$ be the ramification index at $D_i$ as before. Then one can deduce from \cite[Prop.~3.1]{cat-roll09} or simply compute that
\[\sigma(\ka) = \frac 23 |G| (f-1)\sum_{i=1}^me_i\left(1-\frac 1{r_i^2}\right),\]
that is, if $\ka$ is realisable this is the signature of the Kodaira fibration obtained.

\begin{prop}
\label{prop: numerical types signature 4}
The numerical invariants of a virtual double \'etale Kodaira fibration of virtual signature $4$ are 
subject to restrictions, most notably $r_i=2$, which leave the combinations given in Table \ref{tab: virt = 4}.
\end{prop}

We distinguish between {G}raph cases and {C}orrespondence, that is, non-Graph cases.
Note that in this finer classification one of the three cases obtained in Proposition \ref{prop: numerical graph type} splits up into two cases according to the genus of $F$.
The right-hand side of Table \ref{tab: virt = 4} collects some information on realizability which we were able to obtain. 

\begin{table}\caption{Virtual Kodaira fibrations of virtual signature $4$} \label{tab: virt = 4}
\begin{tabular}{lrrcrrll}
 \toprule
type & $b$ &  $f$ & $|G|$ & $g(D_i)$ &  $(d_i, e_i)$ &\begin{minipage}[t]{2.2cm}topological classification \end{minipage} &  realisable\\
 \midrule
$G_1$& $2$& $2$ & $8$ & $2$ & $( 1  ,  1 )$ & Ex. \ref{ex: fpf auto} & no (Rmk. \ref{non-abelian}) \\   
$G_2$& $2$& $2$ & $4$ & $2,2$ & $( 1  ,  1 ), ( 1   ,  1 )$  & Ex. \ref{ex: fpf auto} & no (Ex. \ref{ex: free auto on genus 2}) \\  
$G_3$& $3$& $3$ & $2$ & $3,3$ & $( 1 ,  1 ), ( 1   ,  1  )$ &  Ex. \ref{ex: fpf auto} & no (Ex. \ref{ex: monodromy free involution}, \ref{ex: free auto of order 4 on g = 3})\\
$G_4$&$3$& $2$ & $2$ & $3,3$ & $( 1  ,  2 ), (  1 ,  2 ) $ & Ex. \ref{ex: two graphs 3->2} & no (Ex. \ref{ex: free auto on genus 2}) \\  
\midrule
$C_1$&$2$ & $2$ & $2$ & $5$&  $ (4  ,  4)$  &   &  \\   
$C_2$&$3$& $2$ & $2$ &   $5$&  $ (2  ,  4)$  & (Ex. \ref{ex:C23type}) & 
\\   
$C_3$&$2$& $3$ & $2$ & $5$ & $(  4 ,  2 )$  & (Ex. \ref{ex:C23type}) &  \\
$C_4$&$3$& $3$ & $2$ & $5$ & $( 2  ,  2 )$  & Ex. \ref{ex: C_4 type} &  no (Ex. \ref{ex: double bisection in product of genus 3 curves}, \ref{ex: C4 type with D6 symmetry}) \\ 
$C_5$&$2$& $5$ & $2$ & $5$ & $(  4 ,  1  )$ & & no (Rmk. \ref{non-abelian})\\
$C_6$&$3$& $5$ & $2$ & $5$ & $( 2  ,  1  )$ & Prop. \ref{prop: fpf autos}& no (Rmk. \ref{non-abelian})\\      
$C_7$&$2$& $2$ & $4$ & $3$ & $( 2 ,  2 ) $  & Ex. \ref{ex: double bisection} & no (Ex. \ref{ex: double bisection in product of genus 2}) \\   
$C_8$&$2$& $3$ & $4$ & $3$ & $( 2  ,  1 )$  & Prop. \ref{prop: fpf autos}& no (Rmk. \ref{non-abelian})\\   
$C_{9}$&$2$& $2$ & $2$ & $3,3$ & $( 2 ,  2 ), ( 2  ,  2 ) $  &  &  \\
$C_{10}$&$2$ & $2$ & $2$ & $ 4,2$ & $( 3  ,  3 ), ( 1 ,  1 ) $  &  &  \\   
$C_{11}$&$2$& $3$ & $2$ & $3,3$ & $( 2 ,  1 ), (  2  ,  1  )$  & Ex. \ref{ex: two graphs 3->2} & no (Ex. \ref{ex: free auto on genus 2}) \\   
$C_{12}$&$2$& $2$ & $2$ & $2,2,3$ & $(1,1), ( 1,  1 ), (2  ,  2 ) $  &  &  \\  
 \bottomrule
 \end{tabular}
\end{table}

\begin{proof}
 Setting $\sigma(\ka)=4$ we obtain, as in the proof of Proposition \ref{prop: numerical graph type}, the rough bounds
\[ 2\leq |G|\leq 8, \; 2\leq f\leq 5,\; 1\leq m\leq 4, \; 1\leq e_i\leq 4,\]
from which it is easy to generate a complete list of possibilities. 

It turns out that the ramification has to be  of order $2$ on every component of $D$ and $G$ can be of order $2, 4,8$. Unless $|G| = 8$ the group is necessarily abelian and therefore $\deg (D \to B) = \sum d_i\geq2$. This restricts the possible choices for the genus of $B$, which is implicitly controlled by the fact that $D_i\to B$ is a covering.

As in the proof of Proposition \ref{prop: numerical types signature 4}, the case with $4$ graphs of automorphisms of a curve of genus $2$ can be excluded by \cite{cat-roll09}.
\end{proof}

\begin{rem}
\label{non-abelian}
Now assume that some configuration in Table \ref{tab: virt = 4} is realisable. Then the $G$-cover $f\colon S \to F\times B$ induces a branched cover both of the horizontal and the vertical fibre. Thus if $G$ is abelian, both $\sum_id_i>1$ and $\sum_i e_i>1$. 
This does not hold in cases $C_5,C_6,C_8$ so  that these virtual Kodaira fibrations are not realisable. 

Moreover, if the order of $G$ is not $2$, (cases $G_1,G_2,C_7,C_8$), then any ramified cover $S\to F\times B$ realising $\ka$ factors over an unramified cover $\tilde F\times B \to F\times B$, since the local monodromies at the $D_i$ generate a central subgroup of order $2$.
Hence their configurations give via pullback one of the other configurations with group $G=\IZ/2$, (cases $C_5,C_{11},C_3,C_5$ respectively) and non-realisability of the latter implies non-realisability of the former. 
\end{rem}

\subsection{Outlook on topological existence and classification}
In this section we show that certain cases in Proposition \ref{prop: numerical graph type} and Proposition \ref{prop: numerical types signature 4} actually occur by providing suitable
examples of virtual Kodaira fibrations. 
We will not succeed in a complete classification of cases, because the topology becomes too involved. Let us substantiate this point with some remarks:
\begin{enumerate}
\item 
 One possibility to obtain a configuration with virtual signature 8 is to take a curve of genus 3 with 4 automorphisms with disjoint graphs. This is possible, but the classification of all topological possibilities is a non-trivial task (compare Example \ref{ex: four graphs of autos on genus 3}).
A thorough classification could use the classification of all possible
groups of effective automorphisms on curves of genus at most $48$ by Breuer in  \cite{breuer} (see also \cite{Broughton:91}), 
but still has to investigate mutually disjoint collections of graphs of automorphisms. 
\item  If we stick to virtual signature 4 then configurations of graph type can be classified (see below) but  the cases $C_1$, $C_2$, $C_3$, $C_5$, $C_{10}$ 
might involve some non-Galois covers 
$D\to F$ or $D\to B$
of degree $3$ or $4$ which again are not easily handled. (Compare Example \ref{ex:C23type}.)
\item  When the order of G gets larger it gets more complicated to find all topological equivalence
classes of coverings. In fact these classes correspond to orbits of epimorphisms  $\theta\colon \pi_1(\hat F)\to  G$ 
under the action of a suitable mapping class group.
They are classified by easily calculated invariants in the favourable cases of G being 
abelian \cite{Edmonds}, dihedral \cite{CLP11}, \cite{clp15} or certain split metacyclic groups
\cite{Weigl}. Further then that GAP \cite{MAPCLASSpackage} offers an implementation
to find all orbits for a given group G and given ramification type.
\end{enumerate}
In the following we only classify the pairs $(F\times B, D)$, neglecting for the moment the representation $\theta$. Note that for the monodromy computation only the topological type of a configuration is of importance. However, we want to construct complex manifolds, and thus have to make sure that the topological configuration  can be realised by algebraic curves.
\begin{example}[Types $G_1$, $G_2$, $G_3$]\label{ex: fpf auto}
 Let $B=F$ be a curve of genus $b$ and let $\phi_1, \dots, \phi_m$ be automorphisms of $B$ such that their graphs do not intersect. Precomposing with $\inverse \phi_1$ we may assume that $\phi_1 = \id_B$. Thus if $m=1$ there is only the configuration, namely $(B\times B, \text{diagonal})$.
 
Concerning the case $m=2$ note that the graph of  $\phi_2$ does not intersect the diagonal if and only if $\phi_2$ does not have fixed points. Automorphisms without fixed points on curves of small genus are classified in Appendix \ref{sect: fpf autos}, and  by Proposition \ref{prop: fpf autos}  there is a unique case if $g(B)=2$ and two distinct ones if $g(B)=3$.
\end{example}

\begin{example}[Type $C_7$]\label{ex: double bisection}
 Assume $F$ is a curve of genus $2$ and $D$ is a curve of genus $3$. Then any non-constant map $j_1\colon D\to F$ is an \'etale double cover which has a covering involution $\sigma_1 \in \Aut D$.
 
Now consider a second map $j_2\colon D\to B$ to a curve of genus $2$ corresponding to another involution $\sigma_2$. 
Assuming that $(j_1, j_2)$ embeds $D$ into $F\times B$ is the same as saying that $\phi = \sigma_1\inverse\sigma_2$ does not have a fixed point.
 
By Proposition \ref{prop: fpf autos} such an automorphism $\phi$ has order $2$ or $4$. If $\phi$ has order $2$ then a short computation shows that $\langle \sigma_1, \sigma_2\rangle \isom \IZ/2\times\IZ/2$ acts freely on $D$ which is impossible by the Hurwitz formula. Thus $\phi$ has order $4$ and consequently generates the normal cyclic subgroup in the dihedral group $D_4\isom \langle \sigma_1, \sigma_2\rangle$. By \cite{Broughton:91} $D_4$ acts on a curve of genus $3$ uniquely in such a way that each of $\sigma_1$, $\sigma_2$, and $\phi$ acts freely, and thus this topological configuration is unique. Since this topological action can be realised as a holomorphic group action on a curve $D$, the configuration can also be realised by algebraic curves.
\end{example}

\begin{example}[Type $G_4$]\label{ex: two graphs 3->2}
 Assume we have a curve $B$ of genus $3$ and a curve $F$ of genus $2$ and assume we have two \'etale double covers $\pi_i\colon B\to F$ corresponding to two involutions $\sigma_1$, $\sigma_2$ on $B$. Writing down the condition that $D_1 = \Gamma_{\pi_1}$ and $D_2 = \Gamma_{\pi_2}$ are disjoint in $F\times B$ we find the same condition on $\sigma_1\inverse\sigma_2$ as in Example \ref{ex: double bisection}, that is, topologically this configuration exists and is unique.
 
To see that it also exists in the holomorphic category, consider a curve $F$ of genus $2$ that admits a fixed point free automorphism $\phi$ (which has to be of order $6$ by Proposition \ref{prop: fpf autos}) and consider $\bar D = \Gamma_\id\cup \Gamma_\phi\subset F\times F$.  If $B\to F$ is any \'etale double cover then the pullback $D$ of $\bar D$ to $F\times B$ gives the desired configuration.
\end{example}

\begin{example}[Type $C_4$]\label{ex: C_4 type}
Assume we have a curve $D$ of genus $5$ and two \'etale double covers $j_1\colon D\to B$ and $j_2\colon D\to F$ corresponding to free involutions $\sigma$ and $\tau$, respectively. 
They generate a subgroup of automorphisms of $D$ isomorphic to the dihedral group $D_k$ where $k$ is the order of $\sigma\tau$.
Moreover, if we require that $(j_1,j_2)\colon D\to B\times F$ gives an embedding, then the composition $\sigma\tau$ also must be a free automorphism. By Appendix \ref{sect: fpf autos}, the order of $\sigma\tau$ is even and at most $8$.
In any case, on the quotient of $D$ by the normal cyclic subgroup $N=\langle\sigma\tau\rangle$ the action induced by $\sigma$ (or equivalently $\tau$) must be a free involution.

In case $k=8$ by Appendix \ref{sect: fpf autos} the quotient of $D$ by $N$ has genus $1$ and two branch points of multiplicity $2$.
So the free quotient by an involution is of genus $1$ with a single branch point of multiplicity $2$.
Since $\sigma,\tau$ act freely, only the central element of $D_8$ may be assigned to the branch point. But then there is no way to find
a surjection 
\[ \langle \alpha, \beta, \gamma \mid [\alpha, \beta]\gamma\rangle \to D_8, \qquad \gamma\mapsto (\sigma\circ\tau)^4.\]
In fact dividing by the normal subgroup generated by $(\sigma\circ\tau)^4$ we would get a surjection 
$\langle \alpha, \beta\mid [\alpha, \beta]\rangle \to D_4$ which is impossible since the source is abelian.
Thus the case $k=8$ cannot occur.

In case $k=6$ Appendix \ref{sect: fpf autos} gives two alternatives for the quotient of $D$ by $N$, but in one case the genus is $0$ and thus
the induced action of $\sigma$ cannot be free. In the other case an induced free involution on the quotient of genus $1$ with two branch points of multiplicity
$3$ is possible and give a quotient of genus $1$ with at single branch point. In fact, there is an epimorphism
represented by the map
\[ \langle \alpha, \beta, \gamma \mid [\alpha, \beta]\gamma\rangle \to D_6, \qquad \alpha\mapsto\sigma,
\beta\mapsto\tau, \gamma\mapsto (\sigma\circ\tau)^4.\]
which is unique up to equivalence by \cite{clp15}.

In the case it has order $k=4$, by Appendix \ref{sect: fpf autos} the subgroup generated by $\sigma\tau$ acts either freely or
with $(\sigma\tau)^2$ stabilising four points. The quotient of the first action is of genus two, thus the induced action by $\sigma$ cannot
be free contradicting the assumption. 

For the second action the quotient of $D$ by central involution $(\sigma\tau)^2$ is of genus two with the factor group
$D_4/\langle(\sigma\circ\tau)^2\rangle\cong \mathbb{Z}/2 \times \mathbb{Z}/2$ acting freely, which again is impossible.

In the case $\sigma\tau$ has order $k=2$, $D_k$ is abelian, $\langle\sigma,\tau\rangle\cong\mathbb{Z}/2\times\mathbb{Z}/2$. In this case, the whole group acts freely, hence by the result of Edmonds \cite{Edmonds} there are two topologically distinct cases corresponding to equivalence classes of epimorphisms
\[\langle\alpha_1,\beta_1,\alpha_2,\beta_2\mid [\alpha_1,\beta_1][\alpha_2,\beta_2]\rangle\to\mathbb{Z}/2 \times \mathbb{Z}/2\]
represented by maps
\begin{gather*}
\alpha_1,\alpha_2\mapsto (0,0),\beta_1\mapsto (1,0),\beta_2\mapsto (0,1),\\
\alpha_2,\beta_2\mapsto (0,0),\alpha_1\mapsto (1,0),\beta_1\mapsto (0,1).
\end{gather*}
Thus we have a complete topological classification for configurations of type $C_4$ in Table \ref{tab: virt = 4} obtaining three different topological types.
\end{example}

\begin{example}[Types $C_2$ and $C_3$]\label{ex:C23type}
We consider the case $C_2$; the case $C_3$ is the same with the role of $F$ and $B$ exchanged, and  the case $C_1$ is similar.  
Assume we have a curve $D$ of genus $5$, an \'etale  cover $j_1:D\to F$ of degree $e=4$, and an \'etale cover $j_2\colon D\to B$ 
of degree $d=2$. 
Then $(j_1,j_2)\colon D\to F\times B$ gives an embedding if and only if every pair of fibres of $j_1,j_2$ have at most one point in common.
Let $\tilde D \to F$ be the Galois closure of $D\to F$, with Galois group $\tilde G$ of $\tilde D\to D$ and a lift $\tilde \sigma\in\Aut \tilde D$ of
the involution $\sigma\in\Aut D$ associated to the double cover $D\to B$. To tame this large quantity of possibilities it needs 
some new ideas.

Still, to give a flavour we consider three specific cases:
\begin{description}
\item[dihedral case] Suppose
both projections $j_1,j_2$ are Galois with automorphism $\sigma$ of order two and $\phi$ of order four and $\sigma\circ\phi$ of order two:
Then the group generated is again the dihedral group $D_4$. But all elements have to act freely, not only $\sigma,\phi$ and therefore $\sigma\circ\phi$,
but also $\phi^2,\phi^3$ and therefore $\sigma\circ\phi^2,\sigma\circ\phi^3$. A free $D_4$-action on a genus $5$ curve is not possible.
\item[abelian case] Suppose
both projections $j_1,j_2$ are Galois with automorphism $\sigma$ of order two and $\phi$ of order four and $\sigma\circ\phi$ of order four
and the group generated being abelian:
Then  all seven elements have to act freely which again is not possible.

\item[simplest non-Galois case]
If $D\to F$ is non-Galois, then the image of the monodromy $\pi_1(F)\to S_4$ is a transitive subgroup of order 
$8$, $12$, or $24$.
In the minimal case the image is isomorphic to $D_4$ so the Galois closure $\tilde D\to F$ has $\tilde D$ of genus $9$ 
and the Galois group $D_4$ acts freely.
We further restrict by assuming the lift $\tilde \sigma\in \Aut \tilde D$ to be an involution and to commute with $D_4\subset \Aut\tilde D$.
Then $\tilde D\to D$ is an \'etale Galois cover corresponding to a reflection $\psi\in D_4$.
Therefore $\tilde D \to B$ is an \'etale Galois cover with Galois group generated by $\tilde \sigma$ and $\psi$.
We have thus a subgroup $C_2\times D_4$ in $\Aut\tilde D$ where the second factor and the subgroup generated
by $\tilde\sigma,\psi$ acts freely.

Now consider 
\[\begin{tikzcd}\tilde D \rar{/\psi}  & D \rar{j_1\times j_2} &F\times B.\end{tikzcd}\] 
We will now show that under the current assumptions the map from $D$ to the product is not injective, so we do not get a virtual double-\'etale Kodaira fibration. 

The group  $C_2\times D_4$ cannot act freely on $\tilde D$, since the group order and the Euler number of $\tilde D$ are the same up to sign.
Thus there is an automorphism  which has a fixed point $x$. Since the factor $D_4$ acts freely, the automorphism has the form 
$\tilde\sigma\psi\phi^i$ or $\tilde\sigma\phi^i$. We deduce that $y=\tilde\sigma(x)$ is in the orbit of $x$ under the  $D_4$-action. 
Thus the points $x$ and $\tilde\sigma(x)$  map to the same point in  $F\times B$. 
But since both $\psi\tilde\sigma$ and $\tilde\sigma$ do not have fixed points, 
the points are different and map to different points in $D$ so that the map from $D$ to the product $F\times B$ is not injective.

\end{description}
Accordingly these three cases are not possible but clearly we have not yet covered all possibilities.
\end{example}

\begin{exam}[graph type with virtual signature $8$]\label{ex: four graphs of autos on genus 3}
We will now give two examples of a curve $B=F$ of genus $3$ such that we have four disjoint graphs of automorphisms $D = \Gamma_\id\cup \Gamma_{\phi_1}\cup \Gamma_{\phi_2}\cup \Gamma_{\phi_1\phi_2}\subset F\times B$, where however the group generated by $\phi_1$ and $\phi_2$ is different in each case. Thus in this situation we do no longer have a unique configuration like in Example \ref{ex: two graphs 3->2}. 
The four graphs are disjoint if and only if none of 
\[ \phi_1,\, \phi_2,\, \phi_1\phi_2,
\, \inverse\phi_1\phi_2,\, \inverse\phi_2\phi_1\phi_2\sim \phi_1\]
has fixed points.

For the first example, consider the $D_4$ action from Example \ref{ex: double bisection} and let $\phi_i=\sigma_i$ be the involutions.

For the second example consider the quaternion group $Q_8$ generated by $i$ and $j$ and the Galois cover of an elliptic curve branched over one point given by the surjection
\[ \langle \alpha, \beta, \gamma \mid [\alpha, \beta]\gamma\rangle \to Q_8, \qquad \alpha\mapsto i, \beta\mapsto j, \gamma\mapsto -1.\]
Then $\phi_1=i$ and $\phi_2 = j$, considered as deck transformations, satisfy the required condition.
\end{exam}

\section{Computing realisation genera: examples}\label{sect: computing}
Building on the results in the previous sections we now give several examples. 
Our emphasis is on illustrating different approaches to explicit computations and on checking realisability for many of the virtual Kodaira fibrations from Table \ref{tab: virt = 4}.

Recall that, if we want to determine the realisation signature of a  virtual Kodaira fibration $\ka = (F\times B, D, \theta\colon \pi_1(\hat F)\to G)$ with abelian group $G$ then by Corollary \ref{cor: minimal pullback degree} we have to check that the global extension obstruction $o(\theta)$ vanishes and then we need to compute the index of the stabiliser of $\theta$ with the formula given in Corollary \ref{cor: monodromy action graph type}.

Usually this requires to compute the action of some automorphism on homology, for which we indicate different approaches.

\begin{rem}\label{rem: invariants}
While here our main focus is the signature there are related questions about minimal base and fibre genera. In Table \ref{tab: invariants}, we list for the examples that we construct a complete set of invariants for convenient reference. To unify notation, we consider the examples as double Kodaira fibrations $S\to B_1\times B_2$ and denote a general fibre of $S\to B_i$ by $F_i$. 
\end{rem}

\begin{table}\caption{\protect{Invariants of the Kodaira fibrations constructed (see Remark \ref{rem: invariants} for notation)}}\label{tab: invariants}
 \begin{tabular}{l cc cc cc cc}
\toprule
Example & $g(B_1)$ & $g(F_1)$ &$g(B_2)$ & $g(F_2)$ & $c_2(S)$&$c_1^2(S)$  &$\sigma (S)$ & $c_1^2(S)/c_2(S)$\\
\midrule

\ref{ex: monodromy free involution} ($b=3$) &
$9$ & $6$ & $3$ & $21$ & $160$ & $368$ & $16$ & $2+3/{10}$\\
\ref{ex: free auto on genus 2} 
& $17$& $4$ & $2$ & $49$ & $192$ &  $480$ & $32$ & $2+ 1/2$\\
\ref{ex: free auto of order 4 on g = 3} & 
$33$ & $6$ & $3$  &$81$ & $640$ & $1472$ & $64$ &$2+3/10$\\
\ref{ex: sl23} &
$10$ & $7$ & $2$ & $55$ & $216$ & $576$ & $48$ & $2+2/3$\\
\ref{ex: double bisection in product of genus 2}  &
$9$ & $4$ & $2$ & $25$ & $96$ & $240$ & $16$ & $2+1/2$\\
\ref{ex: double bisection in product of genus 3 curves} (Type 1)&
$9$ & $6$ & $3$ & $21$ & $160$ & $368$ & $16$ & $2+3/10$\\
\ref{ex: double bisection in product of genus 3 curves} (Type 2)&
$17$ & $6$ & $3$ & $41$ & $320$ & $736$ & $32$ & $2+3/10$\\
\ref{ex: C4 type with D6 symmetry} &
$65$ & $6$ & $3$ & $161$ & $1280$ & $2944$ & $128$ & $2+3/10$\\
\ref{ex: four disjoint graphs} &
$5$ & $7$ & $3$ & $13$ & $96$ & $240$ & $16$ & $2+1/2$\\
\bottomrule
 \end{tabular}
 \end{table}

\subsection{Examples of graph type}
Here we consider some virtual Kodaira fibrations where the divisor $D\subset F\times B$ is given as a union of graphs of maps $\phi_i\colon B\to F$. Recall that if $F=B$ then we can (and will)  assume $\phi_1=\id_B$.

\subsubsection{Two graphs of automorphisms}
In this case $D = \Gamma_\id \cup \Gamma_\phi$ for some fixed-point-free automorphism $\phi$. Examples of such are easy to give using the results of Appendix \ref{sect: fpf autos} and to have a chance for small realisation signature it is natural to consider double covers of $F$ branched exactly at the intersection $F\cap D$. 

\begin{exam}[free involutions]\label{ex: monodromy free involution}
 Let $B$ be a curve of odd genus $b=2(q-1)+1\geq 3$ admitting a fixed point free involution $\sigma$. By Appendix \ref{sect: fpf autos}, topologically there is a unique such case.
 
 Then for an appropriate choice of basis, $\sigma$ acts on homology by a block diagonal $2b\times 2b$ matrix 
 \[ \sigma_* = \begin{pmatrix} 1 & 0\\ 0& 1\\ && A\\ &&&\ddots \\&&&&A\end{pmatrix} \text{ where } A=\begin{pmatrix} 0&0&1&0\\0&0&0&1\\1&0&0&0\\0&1&0&0\end{pmatrix},
 \]
 as can be seen by arranging all holes on a line and considering the rotation by $\pi$ through the middle hole.
 We choose a double cover of $F$  branched exactly over the two punctures $F\cap D$. 
 Thus $\iota  = \id+\sigma_*$ (considered modulo $2$) has rank $b-1$ and the kernel of $\iota$ has index $2^{b-1}$. Thus the realisation genus is $\tilde b = 2^{b-1}(b-1)+1$ and we have realisation signature $\tilde \sigma= 2^{b}(b-1)$.
 
 Branched covers of a pullback of this configuration 
were first considered in  \cite{atiyah69, hirzebruch69}. The minimal degree for the pullback for $b=3$ was already found in \cite{BDS}, yielding a double \'etale Kodaira fibration of signature $16$. 
\end{exam}

\begin{exam}[The free automorphism on a curve of genus $2$] \label{ex: free auto on genus 2}
Let $\phi$ be a fixed-point-free automorphism of order $6$ on a curve $B=F$ of genus $2$. By Proposition \ref{prop: fpf autos} there is a unique topological type, which has ramification type $(0\mid 2^2, 3^2)$ (compare Definition \ref{def: ram type}) and is  realised by the surjection
\begin{gather*}
\pi_1^{\text{orb}}(\IP^1; 2,2,3,3) = \langle\gamma_1,\gamma_2,\gamma_3,\gamma_4\mid\gamma_1\gamma_2\gamma_3\gamma_4, \gamma_1^2, \gamma_2^2, \gamma_3^3, \gamma_4^3\rangle\rightarrow \mathbb{Z}/6, \\
(\gamma_1,\gamma_2,\gamma_3,\gamma_4)\mapsto (\phi^3,\phi^3,\phi^2,\phi^4). 
\end{gather*}

We first compute the orbifold fundamental group of the intermediate $\IZ/3$-cover to be
\[\left\langle \gamma_1,\gamma_2, \delta_1 = \gamma_3 \gamma_1 \gamma_3^2,\delta_2 = \gamma_3 \gamma_2 \gamma_3^2,\delta_3 = \gamma_3^2 \gamma_1 \gamma_3,\delta_4 = \gamma_3^2 \gamma_2 \gamma_3
\mid \gamma_i^2, \delta_i^2,  \gamma_1\gamma_2\prod \delta_i\right\rangle\]
and then can  write
\[\pi_1(B) = \left\langle \alpha_1 = \gamma_1\gamma_2, \beta_1 = \delta_1\gamma_2, \alpha_2 = \delta_2\delta_3, \beta_2 = \delta_4\delta_3 \mid \prod [\alpha_i, \beta_i]\right\rangle,\]
that is, $\alpha_1, \dots, \beta_2$ represent a symplectic basis of $H_1(B; \IZ)$.
Since the element  $\gamma_3^2\gamma_1 \in \pi_1^{\text{orb}}(\IP^1; 2,2,3,3)$  maps to  the generator  $\phi$ of $\IZ/6$ we see that for $\eta \in H_1(B; \IZ)$ we have
\[\phi_*(\eta) = \gamma_3^2\gamma_1 \eta \inverse{(\gamma_3^2\gamma_1)} \mod [\pi_1(B), \pi_1(B)].\]
It remains to compute this for the generators of $\pi_1(B)$, which in additive notation gives
$\phi_*(\alpha_1) = \beta_2$, $\phi_*\beta_1 = \beta_1 - \alpha_2$, $\phi_*(\alpha_2) = \beta_1$, $\phi_*\beta_2 = \beta_2 - \alpha_1$. 

We now consider $D = \Gamma_\id \cup \Gamma_{\phi}\subset F\times B$ and let $\theta$ be a homomorphism that defines a double cover of $F$ branched exactly over $D\cap F$. 
The obstruction cocycle from Corollary \ref{cor: minimal pullback degree} vanishes, and thus by Corollary \ref{cor: monodromy action graph type} the minimal realisation degree is the order of the image of $\iota \colon H_1(B; \IZ) \to H_1(B; \IZ/2) $ given by 
\[\id+\phi_* = \id + \begin{pmatrix} 
                      0 & 0 & 0& -1\\
                      0 & 1 & 1 & 0\\
                      0 & -1 & 0 & 0\\
                      1 & 0 & 0& 1
                     \end{pmatrix},
\]
which has full rank. Thus the realisation signature is $\tilde \sigma = 32$ 

Next consider the case where $G = \IZ/4$ as in Case $G_2$ from Table \ref{tab: virt = 4}. 
To compute the realisation signature we check that $o(\theta)$ is again zero and we have to use the above matrix considered with coefficients in $\IZ/4$
and multiplied by $2\in\IZ/4$, the value $\gamma$ takes around the punctures.
Also in this case the image has full rank, but we need to be more careful, since
$\IZ/4$ is not a field. In fact, the image is the kernel of the reduction modulo $2$,
due to the factor $2$ above. Hence the image has cardinality $16$ again,
the configuration is not realisable and the realisation signature is $64$.
We get the same result in case $G=\IZ/2\times\IZ/2$ by a similar argument.

 Notice, that in these cases the ramification at both intersection points has order $2$ and thus any realisation would factor as $S\to \tilde F\times B\to F\times B$ where $\tilde F \to F$ is an \'etale double cover. In other words, any realisation factors over Case $C_{11}$ in Table \ref{tab: virt = 4}. Switching the role of fibre and base we see that a realisation of $C_{11}$ is the same as a realisation of $G_4$.

Consequently,  neither $G_2$ nor $G_4$ nor $C_{11}$ from Table \ref{tab: virt = 4} is realisable. 
\end{exam}

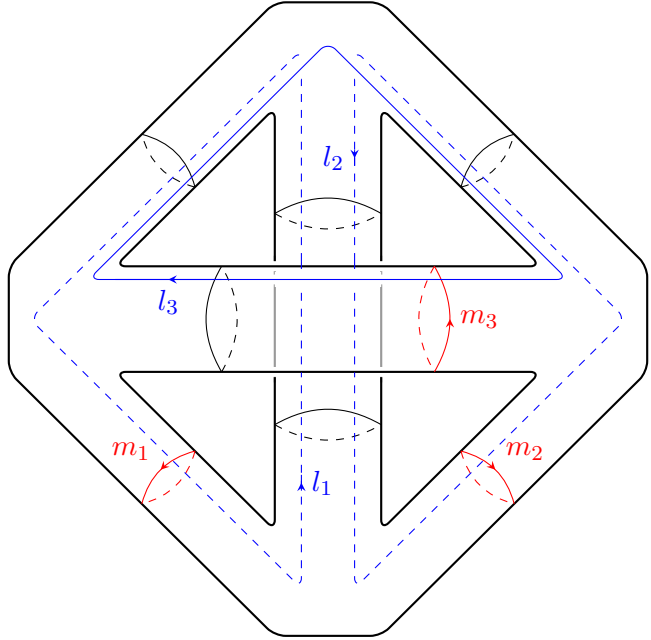
\begin{figure}\caption{Genus three surface with $\mathbb{Z}_4$- symmetry}\label{fg:ord 4 auto on g=3}
\definecolor{qqqqff}{rgb}{0.0,0.0,1.0}
\definecolor{ffqqqq}{rgb}{1.0,0.0,0.0}
\begin{center}
\begin{tikzpicture}
 [border/.style ={thick, black, rounded corners},
m/.style = { red, rounded corners},
l/.style = { blue, rounded corners},
 scale = .7]

\draw[border, white!60!black](1,1.2) -- (1,-1.2) (-1,1.2)-- (-1,-1.2);
\draw[white, line width = .13cm] 
(-1.2, 1) -- ++ (2.4, 0)
(-1.2, .75) -- ++ (2.4, 0)
(-1.2, -1) -- ++ (2.4, 0);

\draw[border] (1,1.1) -- (1,4) -- (4,1) -- (-4,1) -- (-1,4) -- (-1,1.1);
\draw[border] (1,-1.1) -- (1,-4) -- (4,-1) -- (-4,-1) -- (-1,-4) -- (-1,-1.1);
\draw[border] (1, 6) -- (6,1) -- (6,-1) -- (1, -6) -- (-1, -6) -- ( -6, -1) -- (-6, 1)--  (-1, 6) -- cycle;

\begin{scope}[m];
 \draw[mid arrow] (-2.5,-2.5) to[bend right] node[above left] {$m_1$}  (-3.5,-3.5);
 \draw[dashed] (-2.5,-2.5) to[bend left]  (-3.5,-3.5);
 
  \draw[mid arrow] (2.5,-2.5) to[bend left] node[above right] {$m_2$}  (3.5,-3.5);
 \draw[dashed] (2.5,-2.5) to[bend right]  (3.5,-3.5);
 
 \draw[mid arrow] (2,-1) to[bend right] node[right] {$m_3$}  (2,1);
 \draw[dashed] (2,1) to[bend right]  (2,-1);
\end{scope}

\begin{scope}
 \draw 
 (-2, -1) to[bend left] (-2, 1)
   (-1, 2) to[bend left] (1, 2)
 (-1, -2) to[bend left] (1, -2)
 (-2.5, 2.5) to[bend right] ++(-1,1)
  (2.5, 2.5) to[bend left] ++(1, 1);
 \draw[dashed] 
 (-2, -1) to[bend right] (-2, 1)
 (-1, -2) to[bend right] (1, -2)
  (-1, 2) to[bend right] (1, 2)
 (-2.5, 2.5) to[bend left] ++(-1,1)
 (2.5, 2.5) to[bend right] ++(1, 1);
\end{scope}

\begin{scope}[l]
 \begin{scope}[dashed]
 
 \draw[dashed] (-0.5, 1.1) -- ++ ( 0,4) -- ++ (-5.1, -5.1) -- ++(5.1, -5.1) -- ++ (0, 4);
 \draw[dashed] (-0.5, 1.1) -- ++ (0, -2.2);

  \draw[dashed] (0.5, 1.1) -- ++ ( 0,4) -- ++ (5.1, -5.1) -- ++(-5.1, -5.1) -- ++ (0, 4);
 \draw[dashed] (0.5, 1.1) -- ++ (0, -2.2);
\end{scope}
 \draw[mid arrow] (0.5, 3.2) to node[left] {$l_2$} ++(0,-.3);
 \draw[mid arrow] (-0.5, -3.25) to node[right] {$l_1$} ++(0,.3);

 \draw[white, line width = .2cm]
 (0, 5.25)++ (-2, -2) -- (0, 5.25) -- ++(2, -2)
 (-1, 0.75) -- ++(2,0);
 \draw (4.5, 0.75) -- (0, 5.25) -- (-4.5, 0.75) -- cycle;

 \draw[mid arrow] (-2, 0.75) to node[below] {$l_3$} ++(-2,0);
 \end{scope}
\end{tikzpicture}
\end{center}
\end{figure}

\begin{exam}[The free automorphism of order $4$ on a curve of genus $3$]\label{ex: free auto of order 4 on g = 3}
Let $\phi$ be a fixed-point-free automorphism of order $4$ on a curve $B=F$ of genus $3$.  By Proposition \ref{prop: fpf autos} there is a unique topological type, realised by the surjection 
\[\langle\alpha,\beta,\gamma_1,\gamma_2\mid[\alpha,\beta]\gamma_1\gamma_2\rangle\rightarrow \mathbb{Z}/4, \qquad (\alpha,\beta;\gamma_1,\gamma_2)\mapsto (1,\phi;\phi^2,\phi^2).\]
From this, we can find the topological model as in Figure \ref{fg:ord 4 auto on g=3}.
Now consider $D=\Gamma_{id}\cup\Gamma_{\phi}\subset F\times B$ and
consider any homomorphism $\theta$ which defines the double cover of $F$ branched over $F\cap D$.

By Cor \ref{cor: monodromy action graph type} and Prop \ref{prop: tautological graph}, the degree of the minimal pullback is given by the order of $H_1(B;\mathbb{Z})/\ker\iota$, where
\[\iota\colon H_1(B;\mathbb{Z})\rightarrow H_1(F;\IZ/2),\quad\alpha\mapsto (id+\phi_*)\alpha\otimes 1.\]
Since 
\[\phi_*=\begin{pmatrix} 0 & 0 & -1 &  &  &  \\
1 & 1 & 1 &  &  &  \\
0 & -1 & 0 &  &  &  \\
 &  &  & 1 & 0 & -1 \\
 &  &  & 1 & 0 & 0 \\
 &  &  & 1 & -1 & 0  
\end{pmatrix}\]
with respect to the basis $\{m_1,m_2,m_3,l_1,l_2,l_3\}$ of $H_1(F;\mathbb{Z})$ depicted in Figure \ref{fg:ord 4 auto on g=3},
the degree of the minimal pullback is $16$, and hence the realisation genus $\tilde{b}=33$ and the realisation signature $\tilde{\sigma}=64$. Therefore, this example together with Ex. \ref{ex: monodromy free involution} tells us that a virtual Kodaira fibration of $G_3$ type is not realisable. In fact, we have addressed  both types of fee automorphism on a genus $3$ curve
from the classification in Proposition \ref{prop: fpf autos}.
\end{exam}

\subsubsection{Examples with more than two graphs of automorphisms}

\begin{exam}[Triple cover branched over three graphs on a curve of genus $2$]\label{ex: sl23}

We now compute the monodromy and realisation signature in a more complicated case where we have three different graphs of automorphisms. This examples was considered in \cite{cat-roll09} because it has slope 
\[\frac{c_1^2}{c_2} = 2+ \frac{3\sigma}{e}  = 2+2/3,\]
the maximal known. 
In addition it gives an example of a rigid algebraic surface, in the sense that there are no non-trivial deformations.

We first need to construct a triangle curve $B$ of genus  $2$ with automorphism group 
$\mathrm{Sl}_2(\IZ/3)$ and ramification type $(0\mid 3,3,4)$. In loc.\,cit.\ this was achieved by giving a generating vector, however we need a more explicit description to describe the induced action on homology.

Let $\pi_1^{\text{orb}}(\IP^1; 3,3,4) = \langle \gamma_1, \gamma_2, \gamma_3 \mid \gamma_1\gamma_2\gamma_3, \gamma_1^3, \gamma_2^3, \gamma_3^4\rangle$ and consider the sequence of surjections
\[
 \begin{tikzcd}[row sep = small, column sep = small]
\pi_1^{\text{orb}}(\IP^1; 3,3,4)    \rar{\rho} & \mathrm{Sl}_2(\IZ/3) \rar & A_4 \rar & \IZ/3\\
  \gamma_1 \rar[mapsto] & \begin{pmatrix} 0&2\\1& 2\end{pmatrix} \rar[mapsto] & (123) \rar[mapsto] & 1\\
  \gamma_2 \rar[mapsto] & \begin{pmatrix} 0&1\\2& 2\end{pmatrix} \rar[mapsto] & (234) \rar[mapsto] & 2\\
    \gamma_3 \rar[mapsto] & \begin{pmatrix} 2&2\\2& 1\end{pmatrix} \rar[mapsto] & 
(13)(24) \rar[mapsto] & 0\\
 \end{tikzcd}.
\] 
These correspond to a factorisation of $B\to B/ \mathrm{Sl}_2(\IZ/3)$ as a sequence of three (abelian) ramified Galois coverings or equivalently to a chain of subgroups 
\[ \pi_1(B)\triangleleft\Pi_2\triangleleft\Pi_1\triangleleft\pi_1^{\text{orb}}(\IP^1; 3,3,4) \]
Step by step one can compute (by hand and then check with \cite{GAP4})
\begin{align*}
 \Pi_1 & = \left\langle \delta_1 =\gamma_3, \delta_2 = \gamma_1\gamma_3\inverse \gamma_1, \delta_3 = \inverse\gamma_1\gamma_3\gamma_1 \mid\delta_1\delta_2\delta_3, \delta_i^4 \right\rangle,\\
 \Pi_2 & = \left\langle\eta_1 = \delta_1^2, \eta_2 = \delta_2^2, 
 \eta_3 = \delta_3\delta_1^2\inverse\delta_3 , \eta_4 = \delta_3\delta_2^2\inverse\delta_3, \right. \\
 & \qquad\qquad\qquad\qquad\qquad \left. \eta_5 = \delta_3^2, \eta_6 = \delta_1\delta_3^2\inverse\delta_1 \mid \textstyle\prod_i\eta_i, \eta_i^2 \right\rangle,\\
 \pi_1(B) & = \left\langle\alpha_1 = \eta_1\eta_2, \beta_1 = \eta_3\eta_2, \alpha_2 = \eta_4\eta_5, \beta_2 = \eta_6\eta_5 \mid \textstyle\prod_i[\alpha_i, \beta_i] \right\rangle.
 \end{align*}
We compute the action of $\mathrm{Sl}_2(\IZ/3)$ on $H_1(B, \IZ)$  for the generators $g_1 = \rho(\gamma_1)$ and $g_2 = \rho(\gamma_2)$ by calculating the images in homology of the conjugation of  the generators of $\pi_1(B)$ with $\gamma_1$ respectively $\gamma_2$.

This results  in the following matrices describing the action with respect to the symplectic basis  $\alpha_1, \beta_1, \alpha_2, \beta_2$ of $H_1(B, \IZ)$.
\[
 {g_1}_* = \begin{pmatrix} -1 & 0&0&-1\\ -1 & 0& 1& -1\\1& -1& -1& 0\\1&0&0&0\end{pmatrix}
 ,\quad
{g_2 }_* = \begin{pmatrix} -1& -1&0 & -1\\0&0&1&-1\\ 1&-1&-2&2\\ 1&0&-1&1\end{pmatrix}.
 \]
 Note also that $g_3^2  = \rho(\gamma_3)^2 = -\id\in \mathrm{Sl}_2(\IZ/3)$ is the only non-trivial element in the centre and acts on $B$ as the hyperelliptic involution, hence as $-\id$ on homology.
 
We now return to the configuration. All three automorphisms $\phi_2: = -g_1$,  $\phi_3 = -g_2$, and $\phi_2\circ\phi_3^{-1} = g_1g_2^2$ have order $6$ and hence no fixed points and thus we can consider the virtual Kodaira fibration $(B\times B, \Gamma_\id\cup \Gamma_{\phi_2}\cup \Gamma_{\phi_3}, \theta)$ where $\theta$ defines a triple cover branched at the three points.
The global extension obstruction vanishes and thus we need to compute the index of the kernel of 
\[ \id +{\phi_2}_*+{\phi_3}_* = \id -{g_1}_*-{g_2}_* =  
 \begin{pmatrix} 3 & 1& 0& 2\\ 1& 1& -2 & 2 \\ -2 & 2 & 4& -2 \\ -2 & 0 & 1 & 0 \end{pmatrix}
 \]
considered as a map $H_1(B; \IZ) \to H_1(B;\IZ/3)$. This matrix has 
rank $2$, thus the realisation signature is $\tilde\sigma  = 9\cdot \frac{16}{3}=48$ and the realisation genus is $\tilde b = 10$.
\end{exam}

\subsection{Examples of correspondence type}

\begin{figure}\caption{$D_4$ symmetry on a surface of genus 3}
\label{fg:$D_4$ symmetry on g=3}
\definecolor{qqqqff}{rgb}{0.0,0.0,1.0}
\definecolor{ffxfqq}{rgb}{1.0,0.4980392156862745,0.0}
\begin{center}
\begin{tikzpicture}[line cap=round,line join=round,>=triangle 45,x=1.0cm,y=1.0cm, scale=0.8]
\clip(-2.58,-3.6799999999999984) rectangle (7.98,5.260000000000001);
\draw [rotate around={30.924377491634463:(2.089108251606241,1.5162231191771636)},line width=1.2000000000000002pt,color=ffxfqq] (2.089108251606241,1.5162231191771636) ellipse (2.7547202458669746cm and 2.277714563545965cm);
\draw [shift={(-1.04,-0.02)}] plot[domain=-0.10171049718600589:1.0169885413242303,variable=\t]({1.0*4.291373179850469*cos(\t r)+-0.0*4.291373179850469*sin(\t r)},{0.0*4.291373179850469*cos(\t r)+1.0*4.291373179850469*sin(\t r)});
\draw [shift={(5.3,3.1)},dash pattern=on 2pt off 2pt]  plot[domain=3.0129053492520206:4.18502161974507,variable=\t]({1.0*4.117993944617686*cos(\t r)+-0.0*4.117993944617686*sin(\t r)},{0.0*4.117993944617686*cos(\t r)+1.0*4.117993944617686*sin(\t r)});
\draw [line width=1.2000000000000002pt,color=qqqqff] (2.14,1.62)-- (0.94,-2.64);
\draw [line width=1.2000000000000002pt,color=qqqqff] (1.5074634216841138,-0.6255048530213968) -- (1.3956189337232559,-0.4693292771051437);
\draw [line width=1.2000000000000002pt,color=qqqqff] (1.5074634216841138,-0.6255048530213968) -- (1.6843810662767442,-0.5506707228948584);
\draw [color=qqqqff](0.02,-1.2999999999999987) node[anchor=north west] {$\mathbf{\tau}$};
\draw [color=ffxfqq](3.7800000000000002,4.720000000000001) node[anchor=north west] {$\mathbf{\sigma}$};
\draw [color=qqqqff](1.42,-2.0599999999999987) node[anchor=north west] {\textbf{x -axis}};
\draw [line width=1.2000000000000002pt] (2.14,1.62)-- (2.7,1.6);
\draw [line width=1.2000000000000002pt] (3.18,1.6)-- (6.9,1.44);
\draw [line width=1.2000000000000002pt] (5.159889158272156,1.5148434770635633) -- (5.033554346329454,1.3701385521598057);
\draw [line width=1.2000000000000002pt] (5.159889158272156,1.5148434770635633) -- (5.046445653670546,1.6698614478401945);
\draw (5.88,1.120000000000001) node[anchor=north west] {\textbf{y-axis}};
\draw [color=ffxfqq](-0.2,0.780000000000001) node[anchor=north west] {$L_1$};
\draw (3.42,0.960000000000001) node[anchor=north west] {$R_1$};
\draw (0.72,1.8400000000000007) node[anchor=north west] {$L_2$};
\draw [color=ffxfqq](4.82,3.000000000000001) node[anchor=north west] {$R_2$};
\end{tikzpicture}
\end{center}
\end{figure}
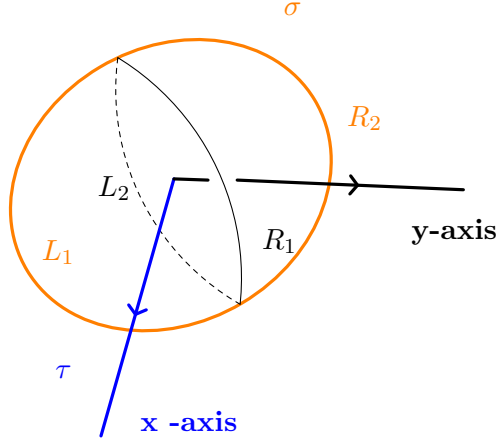

\begin{exam}[Double bisection in a product of two curves of genus two]\label{ex: double bisection in product of genus 2}
In Example \ref{ex: double bisection}, we discussed  the existence and the uniqueness of the configuration corresponding to type $C_7$. 

For a group $G$ of order two  or four, we consider any surjective homomorphism $\theta\colon\pi_1(\hat{F})\to G$ satisfying the ramification condition and the liftability condition. In order to compute the realisation signature of the virtual Kodaira fibration   
$\mathcal{A}=(B\times F, D=(\pi_{\sigma}\times\pi_{\tau})(D_0), \theta\colon\pi_1(\hat{F})\to G)$, we first observe that the global extension obstruction $o(\theta)$ vanishes automatically and thus it remains to calculate the stabiliser of $\theta$. 

For this we describe an explicit topological model of a surface $D$ of genus $3$ with with free involutions $\sigma$ and $\tau$  such that their composition is also free. Such a pair of free involutions appeared in \cite{mitsumatsu2008foliations} to study the self-intersection number of multi-sections of any $\Sigma_g$ bundle over $\Sigma_h$.
Take a graph $\Gamma$ as the intersection of the standard embedded $2$-sphere $S^2$ and $\{(x+y)(x-y)=0\}$ in $\mathbb{R}^3$, and let $D$ be the smooth boundary of a thin regular neighbourhood of $\Gamma$ in $\mathbb{R}^3$. In Figure \ref{fg:$D_4$ symmetry on g=3}, $S^2\cap\{x+y=0\}$ is drawn in orange and $S^2\cap\{x-y=0\}$ is drawn in black. We can think of the rotation of the surface $\Sigma_3$ by $\pi$, denoted by $\sigma$, around the great circle in orange and another $\pi$-rotation of $D$, denoted by $\tau$, around the $x$-axis which is in coordinates $(x,y,z)\mapsto (x,-y,-z)$. 

Under the action of the first involution $\sigma$ on $D$, the torus around the great circle in orange is invariant and rotated by $\pi$ around the core circle, while the other two $1$-handles connecting the regions close to the poles are exchanged. The second involution $\tau$ on $\Sigma_3$ is nothing but the $\pi$-rotation around the axis passing through the middle hole. Evidently, both involutions and their composition are fixed-point-free.

Let $B=D/\sigma$ and $F=D/\tau$, so that $D$ embeds into $F\times B$. 
Now we are ready to compute $\iota\colon H_1(B;\mathbb{Z})\to H_1(F;G)$.
 In Figure \ref{fg:$D_4$ symmetry on g=3}, we fix an orientation of the arc $L_1$ as downward and those of other three arcs as upward. Then we have homology classes of $D$ represented by three longitudes of $\{L_1R_1,L_1R_2,L_1L_2\}$ and three meridians of $\{R_1,R_2,L_2\}$ oriented coherently with the corresponding longitudes.
Now we can take these meridians and longitudes 
of $\{L_1R_1,L_1R_2,L_1L_2\}$ as a basis of $H_1(D;\mathbb{Z})$. Moreover, we have the induced bases of $H_1(B;\mathbb{Z})$ and $H_1(F;\mathbb{Z})$, which are block-wisely $\{L_1R_1,L_1R_2\}$ and $\{R_1,R_2\}$, respectively. 
With respect to these bases, we can compute
\[\pi_{\sigma}^! = \begin{pmatrix} 
1 & 0 & &\\
0&1& &\\
&&1&0\\ 
&&0&2\\ 
1&0&&\\ 
0&1&& \end{pmatrix},
\quad (\pi_{\tau})_*=\begin{pmatrix}  1 &  & 0 &  & 0 &  \\
 & 2 &  & 1 &  & 1\\
 &  & 1 &  & -1 &  \\
&  &  & 1 &  & -1
\end{pmatrix}\]

and hence we get $\iota\colon H_1(B;\mathbb{Z})\to H_1(F;\mathbb{Z}/2)$
\[\iota=(\pi_{\tau})_*\pi_{\sigma}^!\otimes 1= \begin{pmatrix}
1&&0&\\
&3&&2\\
1&&1&\\
&-1&&2
\end{pmatrix} \mod 2
\]
Therefore, we get the index $[\pi_1(B)\colon \Stab_\theta]=8$, the realisation genus $\tilde{b}=9, \tilde{f}=4$ and the realisation signature $\tilde{\sigma}=16$. 

These invariants are shared by example $X_{2,2}$ of \cite{BD} and there are 
enough similarities to conjecture that their surface can be recovered by our construction and 
vice versa.

Finally for a group $G$ of order $4$, by the same argument, $[\pi_1(B)\colon \Stab_\theta]=8$, $\tilde{b}=9$, $\tilde{f}=7$, and $\tilde{\sigma}=32$. Therefore, $C_7$ type is not realisable. 
\end{exam}

\begin{exam}[$C_4$ type: free $\mathbb{Z}/2\times \mathbb{Z}/2$ actions on a curve of genus $5$]\label{ex: double bisection in product of genus 3 curves}
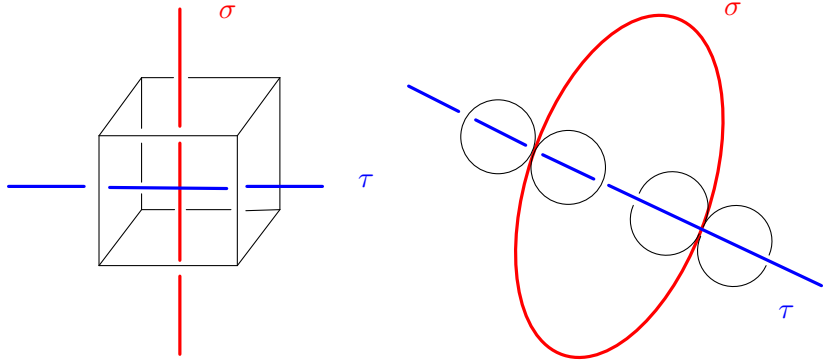
\begin{figure}\caption{$\mathbb{Z}/2\times\mathbb{Z}/2$ symmetries on a surface of genus $5$}
\label{fig: V4 on genus 5}
\definecolor{qqqqff}{rgb}{0.0,0.0,1.0}
\definecolor{ffqqqq}{rgb}{1.0,0.0,0.0}
\begin{tikzpicture}[line cap=round,line join=round,>=triangle 45,x=1.0cm,y=1.0cm,scale=0.7]
\clip(-3.1,-2.7000000000000033) rectangle (13.08,5.0);
\draw (1.5,2.04)-- (2.3,3.14);
\draw (-1.1,2.04)-- (1.5,2.04);
\draw (1.5,-0.42)-- (2.3,0.64);
\draw (1.5,2.04)-- (1.5,-0.42);
\draw (-1.1,2.04)-- (-1.1,-0.42);
\draw (-1.1,-0.42)-- (1.5,-0.42);
\draw (2.3,3.14)-- (2.3,0.64);
\draw [line width=1.2000000000000002pt,color=ffqqqq] (0.42,1.9)-- (0.42,-0.32);
\draw (-1.1,2.04)-- (-0.3,3.14);
\draw [line width=1.2000000000000002pt,color=qqqqff] (-0.9,1.06)-- (1.3,1.04);
\draw [line width=1.2000000000000002pt,color=qqqqff] (1.68,1.08)-- (3.1,1.08);
\draw [line width=1.2000000000000002pt,color=qqqqff] (-2.8,1.08)-- (-1.38,1.08);
\draw [line width=1.2000000000000002pt,color=ffqqqq] (0.42,4.44)-- (0.42,2.22);
\draw [line width=1.2000000000000002pt,color=ffqqqq] (0.42,-0.62)-- (0.42,-2.1);
\draw (-1.1,-0.42)-- (-0.3,0.64);
\draw (-0.3,0.64)-- (-0.3,0.94);
\draw (-0.3,0.6400000000000001)-- (0.32,0.64);
\draw (2.3,0.64)-- (1.66,0.62);
\draw [color=ffqqqq](0.96,4.699999999999999) node[anchor=north west] {$\mathbf{\sigma}$};
\draw [color=qqqqff](3.58,1.4999999999999982) node[anchor=north west] {$\mathbf{\tau}$};
\draw (0.54,0.64)-- (1.38,0.64);
\draw (2.3,3.14)-- (0.58,3.14);
\draw (-0.3,3.14)-- (0.3,3.14);
\draw (-0.3,3.14)-- (-0.3,2.18);
\draw (-0.3,1.9)-- (-0.3,1.16);
\draw [rotate around={71.81634684428694:(8.680000000000001,1.0800000000000027)},line width=1.2000000000000002pt,color=ffqqqq] (8.680000000000001,1.0800000000000027) ellipse (3.362641784924837cm and 1.7290921819621217cm);
\draw [shift={(9.561584318351045,0.4505310420093369)}] plot[domain=2.729217254959007:5.8708099085488,variable=\t]({1.0*0.6750000000000009*cos(\t r)+-0.0*0.6750000000000009*sin(\t r)},{0.0*0.6750000000000009*cos(\t r)+1.0*0.6750000000000009*sin(\t r)});
\draw(7.72,1.44) circle (0.7cm);
\draw(6.4,2.02) circle (0.7cm);
\draw [line width=1.2000000000000002pt,color=qqqqff] (8.48,1.1)-- (12.48,-0.8);
\draw [line width=1.2000000000000002pt,color=qqqqff] (7.24,1.68)-- (8.2,1.22);
\draw [line width=1.2000000000000002pt,color=qqqqff] (5.98,2.34)-- (6.92,1.86);
\draw [shift={(10.874819385901144,0.04443425433514686)}] plot[domain=-0.40454214017280155:2.737050513416992,variable=\t]({1.0*0.6749999999999998*cos(\t r)+-0.0*0.6749999999999998*sin(\t r)},{0.0*0.6749999999999998*cos(\t r)+1.0*0.6749999999999998*sin(\t r)});
\draw [shift={(10.82487947770937,-0.14205802213968344)}] plot[domain=2.717201642179576:5.8587942957693695,variable=\t]({1.0*0.6749999999999995*cos(\t r)+-0.0*0.6749999999999995*sin(\t r)},{0.0*0.6749999999999995*cos(\t r)+1.0*0.6749999999999995*sin(\t r)});
\draw [shift={(9.674819385901145,0.6844342543351482)}] plot[domain=-0.4045421401727989:2.7370505134169942,variable=\t]({1.0*0.6750000000000009*cos(\t r)+-0.0*0.6750000000000009*sin(\t r)},{0.0*0.6750000000000009*cos(\t r)+1.0*0.6750000000000009*sin(\t r)});
\draw [line width=1.2000000000000002pt,color=qqqqff] (4.72,2.98)-- (5.66,2.5);
\draw [color=ffqqqq](10.46,4.76) node[anchor=north west] {$\mathbf{\sigma}$};
\draw [color=qqqqff](11.48,-0.9800000000000026) node[anchor=north west] {$\mathbf{\tau}$};
\end{tikzpicture}
\end{figure}
Let $D$ be a genus $5$ curve with a free action of $G = \IZ/2\times \IZ/2 = \langle \sigma, \tau\rangle$. Let $B =D/\sigma$ and $F = D/\tau$. Then the natural projections embed $D\into F\times B$.

By Example \ref{ex: C_4 type} there are exactly two topologically different such actions. A topological model for both of them is shown in Figure \ref{fig: V4 on genus 5}.

In both cases, the global extension obstruction of the appropriate $\theta$ vanishes because the ramification order is $2$ at each point. So we only need to compute the index $[\pi_1(B): \Stab_{\theta}]$.
First consider the action realised by the epimorphism
\[
\alpha_1,\alpha_2\mapsto (0,0),\beta_1\mapsto (1,0),\beta_2\mapsto (0,1).\]
Representing the surface of genus 5 as the boundary of a tubular neighbourhood of the 1-skeleton of a cube,
this free action of $\IZ/2\times \IZ/2$ is visible as the rotations by $\pi$ around the three axes. Call
$\sigma$ the rotation around the $z$-axis, $\tau$ the rotation around the $y$-axis.
Identify the curve $D$ with the surface of genus $5$ with a row of three holes coming from the West side, the Front side, and the East side, and two more holes from the North and the South. Then we can choose longitudes $l_N,l_S,l_W,l_F,l_E$, oriented from the natural orientation of the cube. The corresponding meridians $m_N,m_S,m_W,m_E$ are meridians of the corresponding four oriented edges of the back face. The meridian $m_F$ however has to be chosen to be the boundary of $2$ dimensional thickening of one of four edges parallel to the $x-$axis. The homology class of $m_F$ does not depend on the choice we made.

Once we fix a basis $\{m_N,l_N,m_S,l_S,m_W,l_W,m_F,l_F,m_E,l_E\}$ of $H_1(D;\mathbb{Z})$, we have the induced homology basis of the quotient with representing cycles obtained as images $\{m_N',l_N'/2,m_S',l_S'/2,m_W',l_W'\}$ for $H_1(B;\mathbb{Z})$ and $\{m_N'',l_N'',m_W'',l_W''/2,m_E'',l_E''/2\}$ for $H_1(F;\mathbb{Z})$.

 With respect to these bases, we can compute
\begin{align*} \pi_{\sigma}^! &= \begin{pmatrix} 
2 & 0 & &&&&-1&0&&\\
0&1&&&&&0&0&&\\
&& 2 & 0 &&&-1&0&& \\ 
&& 0 & 1 &&&0&0&& \\ 
&&&& 1 & 0 &-1&0&1&0 \\ 
&&&& 0 & 1 &0&0&0&1 \\ 
                   \end{pmatrix}^T,
                   \\
                   {\pi_{\tau}}_* & = \begin{pmatrix} 
1 & 0 &1 &0&&&0&0&&\\
0&1&0&1&&&0&-1&&\\
&& &  &1&0 &0&0&&\\ 
&&  &  &0&2 &0&-1&&\\ 
&&&&&& 0 & 0 &1&0 \\ 
&&&& &&0 & -1 &0&2 \\ 
                   \end{pmatrix}
\end{align*}
and hence we get $\iota\colon H_1(B;\mathbb{Z})\to H_1(F;\mathbb{Z}/2)$
\[\iota={\pi_{\tau}}_*{\pi_{\sigma}}^!\otimes 1 = \begin{pmatrix}2 &  &2 &&&\\
&1&&1&&\\
&& & &1& \\ 
&& &  &&2 \\ 
&& & &1& \\ 
&&&  &&2 
\end{pmatrix}
\mod 2\]
Therefore we get the index $[\pi_1(B): \Stab_{\theta}]=4$, the realisation genus $\tilde{b}=9$, $\tilde{f}=6$ and the realisation signature $\tilde{\sigma}=16$.

Now move to the second action realised by the epimorphism
\[\alpha_2,\beta_2\mapsto (0,0),\alpha_1\mapsto (1,0),\beta_1\mapsto (0,1).
\]
From Figure \ref{fig: V4 on genus 5}, we can choose a basis of $H_1(D;\mathbb{Z})$ as meridians and longitudes from the core torus, and then from four handles, precisely, the Front-left, the Front-right, the Back-left, and the Back-right. Then we have the induced bases in $H_1(B;\mathbb{Z})$ and $H_1(F;\mathbb{Z})$, which are block-wise \{Core, Front-left, Front-right\} and \{Core, Front-left, Back-left\}, respectively. With respect to these bases, we can compute 
\begin{align*}
\pi_{\sigma}^!& = \begin{pmatrix} 
1 & 0 & &&&&&&&\\
0&2&&&&&&&&\\
&& 1 & 0 &&&1&0 \\ 
&& 0 & 1 &&&0&1 \\ 
&&&& 1 & 0 &&&1&0 \\ 
&&&& 0 & 1 &&&0&1 \\ 
                   \end{pmatrix}^T, \\
                   {\pi_{\tau}}_*& = \begin{pmatrix} 
1 & 0 & &&&&&&&\\
0&2&&&&&&&&\\
&& 1 & 0 &1&0 \\ 
&& 0 & 1 &0&1 \\ 
&&&&&& 1 & 0 &1&0 \\ 
&&&& &&0 & 1 &0&1 \\ 
                   \end{pmatrix}
\end{align*}
and hence $\iota\colon H_1(B;\mathbb{Z})\to H_1(F;\mathbb{Z}/2)$
\[\iota={\pi_{\tau}}_*\pi_{\sigma}^!\otimes 1 = \begin{pmatrix}1 & 0 & &&&\\
0&4&&&\\
&& 1 & 0 &1&0 \\ 
&& 0 & 1 &0&1 \\ 
&& 1 & 0 &1&0 \\ 
&&0 & 1 &0&1 
\end{pmatrix}
\mod 2\]
Therefore we get the index $[\pi_1(B)\colon \Stab_{\theta}]=8$, the realisation genus $\tilde{b}=17, \tilde{f}=6$ and the realisation signature $\tilde{\sigma}=32$.
\end{exam}

\begin{exam}[$C_4$ type: $D_6$ symmetry on a curve of genus $5$]\label{ex: C4 type with D6 symmetry}
 \begin{figure}\caption{$D_6$ symmetry on a curve of genus $5$}
\label{fig: d6 on genus 5}
\definecolor{qqqqff}{rgb}{0.0,0.0,1.0}
\definecolor{ffxfqq}{rgb}{1.0,0.4980392156862745,0.0}
\begin{tikzpicture}[line cap=round,line join=round,>=triangle 45,x=1.0cm,y=1.0cm]
\clip(-2.72,-1.3000000000000038) rectangle (6.96,5.419999999999998);
\draw [rotate around={0.0:(1.95,2.0)},line width=1.2000000000000002pt,color=ffxfqq] (1.95,2.0) ellipse (2.7382903593571952cm and 2.0003334952323715cm);
\draw [shift={(0.08,2.0)}] plot[domain=-0.8190604466528733:0.8190604466528735,variable=\t]({1.0*2.7382903593571952*cos(\t r)+-0.0*2.7382903593571952*sin(\t r)},{0.0*2.7382903593571952*cos(\t r)+1.0*2.7382903593571952*sin(\t r)});
\draw [shift={(3.82,2.0)}] plot[domain=2.3225322069369194:3.9606531002426664,variable=\t]({1.0*2.7382903593571952*cos(\t r)+-0.0*2.7382903593571952*sin(\t r)},{0.0*2.7382903593571952*cos(\t r)+1.0*2.7382903593571952*sin(\t r)});
\draw [line width=1.2000000000000002pt,color=qqqqff] (0.7698959035178234,1.7933480840855458)-- (-2.5,0.66);
\draw [line width=1.2000000000000002pt,color=qqqqff] (-0.9784347079139911,1.1873755371666503) -- (-0.9141751793362414,1.368402366633902);
\draw [line width=1.2000000000000002pt,color=qqqqff] (-0.9784347079139911,1.1873755371666503) -- (-0.8159289171459346,1.0849457174516435);
\draw [line width=1.2000000000000002pt,color=qqqqff] (2.942886195130181,2.4182568313490727)-- (4.2,2.86);
\draw [line width=1.2000000000000002pt,color=qqqqff] (4.737035492703126,3.013379089161871)-- (6.427684679455587,3.6013462033215404);
\draw [color=ffxfqq](1.84,5.119999999999998) node[anchor=north west] {$\mathbf{\sigma}$};
\draw [color=qqqqff](-1.54,0.7999999999999969) node[anchor=north west] {$\mathbf{\tau}$};
\draw (1.52,1.479999999999997) node[anchor=north west] {$R_1$};
\draw (2.92,1.9399999999999973) node[anchor=north west] {$R_2$};
\draw [color=ffxfqq](4.92,2.5999999999999974) node[anchor=north west] {$R_3$};
\draw [color=ffxfqq](-1.6600000000000001,2.5599999999999974) node[anchor=north west] {$L_1$};
\draw (0.04,3.059999999999998) node[anchor=north west] {$L_2$};
\draw (2.98,3.739999999999998) node[anchor=north west] {$L_3$};
\draw [shift={(2.7,3.1)},dash pattern=on 1pt off 1pt]  plot[domain=0.7366563120875167:2.2398393249854704,variable=\t]({1.0*1.1610340218959991*cos(\t r)+-0.0*1.1610340218959991*sin(\t r)},{0.0*1.1610340218959991*cos(\t r)+1.0*1.1610340218959991*sin(\t r)});
\draw [shift={(1.62,2.38)},dash pattern=on 1pt off 1pt]  plot[domain=-1.4330383793695622:0.10697263946934481,variable=\t]({1.0*2.403099571080058*cos(\t r)+-0.0*2.403099571080058*sin(\t r)},{0.0*2.403099571080058*cos(\t r)+1.0*2.403099571080058*sin(\t r)});
\draw [shift={(1.24,3.0)},dash pattern=on 1pt off 1pt]  plot[domain=0.9535478242624417:2.356194490192345,variable=\t]({1.0*1.2266894886986737*cos(\t r)+-0.0*1.2266894886986737*sin(\t r)},{0.0*1.2266894886986737*cos(\t r)+1.0*1.2266894886986737*sin(\t r)});
\draw [shift={(2.16,2.3)},dash pattern=on 1pt off 1pt]  plot[domain=3.5782198134033343:4.609809127966395,variable=\t]({1.0*2.3174123500145587*cos(\t r)+-0.0*2.3174123500145587*sin(\t r)},{0.0*2.3174123500145587*cos(\t r)+1.0*2.3174123500145587*sin(\t r)});
\draw [shift={(2.02,2.38)},dash pattern=on 1pt off 1pt]  plot[domain=2.743964662067664:3.4224303782109957,variable=\t]({1.0*2.169239498073*cos(\t r)+-0.0*2.169239498073*sin(\t r)},{0.0*2.169239498073*cos(\t r)+1.0*2.169239498073*sin(\t r)});
\draw [shift={(2.08,2.56)},dash pattern=on 1pt off 1pt]  plot[domain=0.21182754748141744:0.4521538622857757,variable=\t]({1.0*1.9025246384738355*cos(\t r)+-0.0*1.9025246384738355*sin(\t r)},{0.0*1.9025246384738355*cos(\t r)+1.0*1.9025246384738355*sin(\t r)});
\end{tikzpicture}
\end{figure}
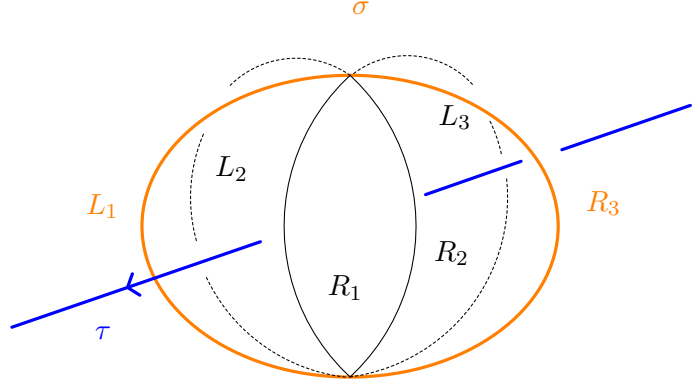
Let $D$ be a genus $5$ curve with an action of dihedral group $G=D_6=\langle\sigma,\tau\rangle$ realised by the equivalence class of an epimorphism \[\langle\alpha,\beta,\gamma\mid[\alpha,\beta]\gamma\rangle\to D_6,\quad(\alpha,\beta;\gamma)\mapsto(\sigma,\tau;(\sigma\tau)^4).\] 
We can find its topological model as in Figure \ref{fig: d6 on genus 5}. Consider a graph $\Gamma$ in $\mathbb{R}^3$ consisting of $S^2\cap\{y+x=0\}$ in orange and $S^2\cap\{x\cdot(y-x)=0\}$ in black. Now realise a surface of genus $5$ as the smooth boundary of a thin regular neighbourhood of $\Gamma$ in $\mathbb{R}^3$. We take a free involution $\sigma$ as the rotation by $\pi$ about the circle in orange, and another free involution $\tau$ as the rotation by $\pi$ about the $x$- axis in blue. Then they generate the dihedral group $D_6$ because $\sigma\tau$ has order $6$. 

Let $B=D/\sigma$ and $F=D/\tau$.
From Figure \ref{fig: d6 on genus 5}, we can choose a basis of $H_1(D;\mathbb{Z})$ as meridians and longitudes of $\{L_1R_1,L_1R_2,L_1R_3,L_1L_3,L_1L_2\}$ as in Example \ref{ex: double bisection in product of genus 2}. Then we have the induced bases of $H_1(B;\mathbb{Z})$ and $H_1(F;\mathbb{Z})$, which are block-wise $\{L_1R_1,L_1R_2,L_1R_3\}$ and  $\{R_1,R_2,R_3\}$, respectively. With respect to these bases, 
we can compute
\[
 \pi_{\sigma}^!   = \begin{pmatrix} 
1 & 0 & &&&&&&1&0\\
0&1&&&&&&&0&1\\
&& 1 & 0 &&&1&0 \\ 
&& 0 & 1 &&&0&1 \\ 
&&&& 1 & 0 &&&& \\ 
&&&& 0 & 2 &&&& \\ 
                   \end{pmatrix}^T,
                   \]
                   \[
                   {\pi_{\tau}}_* = \begin{pmatrix} 
1 & 0 & 0&&0&&0&&0&\\
0&2&&1&&1&&1&&1\\
&& 1 & 0 &&&&&-1&0 \\ 
&& 0 & 1 &&&&&0&-1 \\ 
&&&& 1 & 0 &-1&0&& \\ 
&&&&0 & 1 &0&-1&& \\ 
                   \end{pmatrix}
\]
and hence $\iota\colon H_1(B;\mathbb{Z})\to H_1(F;\mathbb{Z}/2)$ is given by 
\[\iota={\pi_{\tau}}_*\pi_{\sigma}^!\otimes 1 = \begin{pmatrix}1 & 0 & &&&\\
0&3&&2&&2\\
1&0& 1 & 0 && \\ 
0&1& 0 & 1 && \\ 
&& 1 & 0 &1&0 \\ 
&&0 & 1 &0&2 
\end{pmatrix}
\mod 2.\]
Therefore we get the index $[\pi_1(B):\Stab_{\theta}]=32$, the realisation genus $\tilde{b}=65$ and the realisation signature $\tilde{\sigma}=128$.
\end{exam} 

\begin{exam}[four disjoint graphs in a product of genus 3 curves]\label{ex: four disjoint graphs}
In Example \ref{ex: four graphs of autos on genus 3} we constructed a virtual Kodaira fibration $(B\times F, \Gamma_{id}\cup\Gamma_{\sigma}\cup\Gamma_{\tau}\cup\Gamma_{\sigma\tau}, \theta)$, where $B=F$ is a curve of genus $3$ with the automorphism group $D_4=\langle\sigma,\tau\rangle$ such that each of $\sigma, \tau, \sigma\tau$ acts freely and $\theta$ defines a $2$-fold cyclic cover of $F$ branched over four points. First, we can observe the global extension obstruction $o(\theta)$ vanishes in $\mathbb{Z}/2$, and then
using a topological model for such a $D_4$ action in Example \ref{ex: double bisection in product of genus 2}, we can compute $\iota\colon H_1(B;\mathbb{Z})\to H_1(F;\mathbb{Z}/2)$. 

With respect to the basis of $H_1(B;\mathbb{Z})=H_1(F;\mathbb{Z})$ chosen as meridians and longitudes 
of $\{L_1R_1,L_1R_2,L_1L_2\}$ as in Example \ref{ex: double bisection in product of genus 2}, 
\[\sigma_* = \begin{pmatrix} &  & &&1&0\\
&&&&0&1\\
&& 1 & 0 && \\ 
&& 0 & 1 && \\ 
1&0&  &  && \\ 
0&1&&  && 
\end{pmatrix},\quad \tau_*= \begin{pmatrix} 1& &0 &&0&\\
&1&&1&&1\\
1&&  &  &-1& \\ 
&0&  &  && -1\\ 
1&&-1&  && \\ 
&0&&-1&& 
\end{pmatrix}\]
\[(\sigma\tau)_*=\begin{pmatrix}  1& &-1 &&&\\
&0&&-1&&\\
1&&  &  &-1& \\ 
&0&  &  && -1\\ 
1&&0&  &0& \\ 
&1&&1&&1 
\end{pmatrix}\]

Hence, $\iota\colon H_1(B;\mathbb{Z})\to H_1(F;\mathbb{Z}/2)$ is given by \[ \iota=id+\sigma_*+\tau_*+(\sigma\tau)_*=\begin{pmatrix}
1&&1&&1&\\
&0&&0&&0\\
0&&0&&0&\\
&0&&0&&0\\
1&&1&&1&\\
&0&&0&&0
\end{pmatrix}\mod 2
\]
\end{exam}
Therefore, the degree of the minimal pullback is $2$ and the realisation signature $\tilde\sigma=16$.

\appendix
\section{Automorphisms without fixed points on curves of small genus}\label{sect: fpf autos}
In this section we classify automorphisms without fixed points on curves  of genus up to nine. 
three cad
Let us fix some notation for this section: let $B$ be a curve of genus $2\leq b\leq 9$ and assume that $\phi\in \Aut B$ acts on $B$ without fixed points. Let $d$ be the order of $\phi$ and let  $q$ be the genus of the quotient curve $B/\langle \phi\rangle$. 
We now first classify the possible ramification types of the $\IZ/d$-covers $B\to B/\langle \phi\rangle $ and then proceed to classify topological types.

\begin{prop}\label{prop: ramification fpf autos} The ramification types of a fixed-point-free automorphisms $\phi$ of order $d$ on a curve of genus $b\leq 9$ are exactly the following:
\begin{center}
 \begin{tabular}{ccc}
 \toprule
 genus $b$ & $d=\ord \phi$ & ramification type\\
 \midrule
$9$ & $2$ & $(5\mid  -)$\\
$9$ & $4$ & $(3\mid  -)$\\
$9$ & $8$ & $(2\mid  -)$\\
$9$& $4$ & $(2\mid  2^4)$ \\
$9$ & $16$ & $(1\mid  2^2)$\\
$9$ & $12$ & $(1\mid 3^2)$\\
$9$ & $10$ & $(1\mid 5^2)$\\
$9$ & $8$ & $(1\mid 2^4)$\\
$9$ & $8$ & $(1\mid 2,4^2)$\\ 
$9$ & $6$ & $(1\mid 3^4)$\\
$9$ & $4$ & $(1\mid  2^8)$  \\
$9$ & $12$ & $(0 \mid 2,3^2, 4^2)$\\
$9$&$10$ & $(0\mid 2^4, 5^2)$\\
$9$&$6$ & $(0\mid2^4, 3^4)$\\
\midrule
$8$ & $7$ & $(2\mid -)$\\
$8$ & $14$ & $(1\mid 2^2)$\\
$8$ & $6$ & $(1\mid  2^2, 3^2)$\\
$8$&$18$ & $(0\mid 2^2, 9^2)$\\
$8$&$15$ & $(0\mid 3^2, 5^2)$\\
$8$&$12$ & $(0\mid 4^2, 6^2)$\\
$8$&$10$ & $(0\mid 2^2,5^3)$\\
$8$&$6$ & $(0\mid 2^2,3^5)$\\
$8$&$6$ & $(0\mid 2^6, 3^2)$\\
\midrule
$6$ & $5$ & $(2\mid -)$\\
$6$ & $10$ & $(1\mid 2^2)$\\
$6$&$14$ & $(0\mid 2^2, 7^2)$\\
$6$&$12$ & $(0\mid 4^2, 3^2)$\\
$6$&$6$ & $(0\mid 2^2, 3^4)$\\
\bottomrule
\end{tabular}
\quad
 \begin{tabular}{ccc}
 \toprule
 genus $b$ & $d=\ord\phi$ & ramification type\\
\midrule
$7$ & $2$ & $(4\mid -)$\\
$7$ & $3$ & $(3\mid -)$\\
$7$ & $6$ & $(2\mid -)$\\
$7$& $4$ & $(2\mid  2^2)$  \\
$7$ & $12$ & $(1\mid 2^2)$ \\
$7$ & $9$ & $(1\mid  3^2)$\\
$7$ & $8$ & $(1\mid  4^2)$\\
$7$ & $6$ & $(1\mid 2^4)$\\
$7$ & $6$ & $(1\mid 3^3)$\\
$7$ & $4$ & $(1\mid  2^6)$  \\
$7$ & $12$ & $(0 \mid 3, 4^2, 6)$\\
$7$ & $6$ & $(0 \mid 2^4, 3^3)$\\
\midrule
$5$ & $2$ & $(3\mid -)$\\
$5$ & $4$ & $(2\mid -)$\\
$5$ & $8$ & $(1\mid  2^2)$\\
$5$ & $6$ & $(1\mid 3^2)$\\
$5$ & $4$ & $(1\mid  2^4)$  \\
$5$&$6$ & $(0\mid 2^4, 3^2)$\\
\midrule
$4$ & $3$ & $(2\mid -)$\\
$4$ & $6$ & $(1\mid 2^2)$\\
$4$&$6$ & $(0\mid 2^2,3^3)$\\
$4$&$10$ & $(0\mid 2^2, 5^2)$\\
\midrule
$3$ & $2$ & $(2\mid -)$\\
$3$ & $4$ & $(1\mid  2^2)$  \\
\midrule
$2$&$6$ & $(0\mid 2^2, 3^2)$\\
\bottomrule
 \end{tabular}
\end{center}
\end{prop}
We split the proof in several Lemmas. The first step is a simple application of the Hurwitz formula.
\begin{lem}
 If $B\to B/\langle\phi\rangle$ is \'etale then $(b-1) = d(q-1)$, which gives the unramified cases listed in the table.
\end{lem}

We now need to consider the ramified case. The ramified cover $B\to B/\langle\phi\rangle$ corresponds to a surjection
\[\eta\colon\pi_1(B/G\setminus \{P_1, \dots, P_m\}) = \langle \alpha_1, \dots , \beta_q, \gamma_1, \dots, \gamma_m\mid \prod [\alpha_i, \beta_i]\prod \gamma_i = 1\rangle \onto\IZ/d,\]
where $m$ is the number of branch points. Denoting   $a_i = \eta(\gamma_i)$ the surjection $\eta$ gives a tuple of elements $\underline a = (a_1, \dots, a_m)\in \left(\IZ/d\setminus\{0\}\right)^m$ such that $\sum_{i=1}^m a_i =0$ (in particular $m\geq 2$). The ramification properties of $\phi$ are encoded in the $a_i$ and thus we can now  reduce the classification of fixed-point-free automorphisms to tuples with certain properties. We are mainly interested in the ramification type of the cover induced by $\eta$, that is, in the data $(q\mid \ord a_1, \dots, \ord a_m)$.

\begin{lem}\label{lem: tuples for auto} Let $d, m\geq2$ and $q\geq 0$ be integers and  $\underline a = (a_1, \dots, a_m)\in \left(\IZ/d\setminus\{0\}\right)^m$ such that $\sum_{i=1}^m a_i =0$.
\begin{enumerate}
\item There exists a surjection $\eta$ giving rise to the tuple $\underline a$ if either $q>0$ or $q=0$ and the $a_i$ generate $\IZ/d$.
\item Assume \refenum{i} holds and let $H_i  = \langle a_i\rangle$. Then a generator of $\IZ/d$ corresponds to a fixed-point-free automorphism $\phi$ on the ramified cover induced by $\eta$ if and only if for all $i$ we have $H_i\subsetneq \IZ/d$, that is, $1<\ord a_i <d$.
\end{enumerate}
\end{lem}
\begin{proof}
 For the first item note that since $\IZ/d$ is abelian we can prescribe arbitrary images for $\alpha_i$ and $\beta_i$ to define $\eta$ as long as $\sum_i a_i = \eta(\prod \gamma_i) = 0$. If $q>0$ then setting $\eta(\alpha_1)=1$ makes sure that $\eta$ is a surjection. If $q=0$ then the image of $\eta$ is the subgroup generated by the $a_i$ and thus the claim follows.
 
 Considering the elements of $\IZ/d$ as automorphisms of the ramified cover, an $a$  has fixed points if and only if $a\in H_i$, which gives the second item.
\end{proof}

\begin{proof}[Proof of Proposition \ref{prop: ramification fpf autos}.]  We are now ready to complete the classification of possible ramification types by studying the existence of data $d, m, q, \underline a$ satisfying the conditions of Lemma \ref{lem: tuples for auto}. 
\begin{description}[style = nextline]
 \item[Step 1 --- Numerical restrictions] Let us record some direct consequences of Lemma \ref{lem: tuples for auto}:
 \begin{gather}
   \ord a_i |d \text{ and } 2\leq \ord a_i < d \label{eq: order ai}\\
   d \text{ is not prime and thus } d\geq 4 \label{eq: d not prime}  \\
   m \geq2 \text{ and } m=2 \implies a_1 = -a_2 \label{eq: m=2}
 \end{gather}
 To limit the number of cases to consider we compute  $b = g(B)\leq 9$ via the Hurwitz formula:
 \begin{equation}\label{eq: Hurwitz}
  16\geq 2b-2 = d \left(2q-2 + \sum_i \frac {\ord a_i - 1}{\ord a_i}\right).
 \end{equation}
 Substituting the limit cases of \eqref{eq: order ai} and \eqref{eq: d not prime} and \eqref{eq: m=2} we get
 \begin{equation}\label{eq: rough estimate} 16 \geq d\left(2q-2+\frac m2\right)\geq d\left(2q-1\right)\geq 8q-4 \end{equation}
 and thus $q\leq 2$ and $d\leq 16$ unless $q=0$.
\item[Step 2 --- $q=2$]
 If $q=2$ then \eqref{eq: rough estimate}  and \eqref{eq: d not prime} give $d=4$ and thus by \eqref{eq: order ai} $\ord a_i = 2$. Now \eqref{eq: Hurwitz} becomes
 \[ 16 \geq 4\left(2+\frac m2\right)\]
 which leaves the possibilities $m=2$ and $m=4$, because three elements of order $2$ in $\IZ/4$ cannot sum up to $0$.
\item[Step 3 --- $q=1$]
If $q=1$ then $d\in \{4,6,8,9,10,12,14,15,16\}$ by \eqref{eq: d not prime} and \eqref{eq: rough estimate} and all cases can be treated in a similar fashion.

Let us exemplify this for $d=16$: the ramification order can only be $2, 4, 8$ thus \eqref{eq: Hurwitz} becomes
 \[16 \geq  16\left(\frac {m_2}2+ \frac{3m_4}{4}+\frac{7m_8}8\right) = 8m_2+12m_4+14m_8\]
 where $m_r$ denotes the number of elements $a_j$ of order $r$. Since $m\geq 2$ there is only one case. 
 \item[Step 4 --- Preparations for $q=0$]
If $q=0$ then by Lemma \ref{lem: tuples for auto} the elements $a_i$ generate the group $\IZ/d$. Note that if $d=p^\nu$ is a prime power and $\IZ/d$ is generated by $a_1, \dots, a_m$ then at at least one of the $a_i$ maps to a generator. In particular, if $q=0$ then $d$ cannot be a prime power by \eqref{eq: order ai}.  

Consider a decomposition in primes $d =\prod_{i=1}^k p_j^{\nu_j}$ and for each $j$ the projection $\pi_j\colon \IZ/d \onto \IZ/p_j^{\nu_j}$. Since the relations $\sum_i a_i $ remains valid after projection at least two of the $a_i$ map to a generator of $\IZ/p_j^{\nu_j}$.  On the other hand, no $a_i$ can map to a generator under all projections, because in this case it would be a generator for the whole group $\IZ/d$. 

With this information we will now distinguish the cases according to the prime decomposition of $d$.

\item[Step 5 ---  $q=0$, $d = p_1^{\nu_1}p_2^{\nu_2}$]
If $d$ is the product of two prime powers then by Step 4 we have  $m\geq 4$ and we can bound \eqref{eq: Hurwitz} from below by the case where $m=4$ and $\ord a_1 = \ord a_2 = p_1^{\nu_1}$ and $\ord a_3 = \ord a_4 = p_2^{\nu_2}$. This gives 
\[16 \geq p_1^{\nu_1}p_2^{\nu_2}\left(-2+ 2\frac{p_1^{\nu_1}-1}{p_1^{\nu_1}}+ 2\frac{p_2^{\nu_2}-1}{p_2^{\nu_2}}\right) \iff 9\geq (p_1^{\nu_1}-1)(p_2^{\nu_2}-1),\]
which implies $d\in \{6,10,12,14,15,18\}$. We only treat the case $d=18$; the computations are similar in the other cases.

Again we denote by $m_i$ the number elements of order $i$. By \eqref{eq: order ai} we can write  \eqref{eq: Hurwitz} as
 \begin{align*}
 & 16 \geq  18\left(-2+\frac {m_2}2+ \frac{2m_3}{3}+ \frac{5m_6}{6}+\frac{8m_9}9\right) \\
\iff\qquad & 52 \geq 9m_2+12m_3+15m_6+16m_9
\end{align*}
By Step 4 we have (at least) two elements of order divisible by $9$ and two elements of order divisible by $2$ and thus we have the further restrictions
\[m_9\geq 2\text{ and } m_2+m_6 \geq 2.\]
This leaves $m_9 = m_2 = 2$ and $m_3 = m_6 = 0$ as the only possibility.
\item[Step 6 ---  $q=0$ and $d$ has at least three different  prime divisors]
Of all cases in which the decomposition $d =\prod_{i=1}^k p_j^{\nu_j}$ contains at least three different primes, the smallest possible value for $b$ in  \eqref{eq: Hurwitz} occurs if the primes, their multiplicities and the number of ramification points is as low as possible.  Thus the minimal possibility is $\IZ/d = \IZ/2\times \IZ/3\times \IZ/5$ and $m=3$. Then up to isomorphism the only choice is $a_1 = (1,1,0)$, $a_2 = (1,0,1)$, $(0,-1,-1)$, which by \eqref{eq: Hurwitz} leads to a curve of genus $11$. Therefore we can exclude this case altogether since we are interested in curves of genus at most $9$.
\end{description}
\end{proof}
The topological classification of finite order automorphisms of topological surfaces was studied by Nielsen in \cite{nielsen}: let $\phi$  be an automorphism of $B$ of order $d$ and ramification data $(q\mid r_1, \dots, r_m)$.  The quotient map   $B\to B/\langle\phi\rangle$ induces an exact sequence on orbifold fundamental groups 
\[ 1\to \pi_1(B) \to \pi_1^{\text{orb}}(B/\phi; r_1,\dots,r_m)\overset{\rho}\to \langle\phi\rangle \isom \IZ/d \to 0,
 \]
where
\[\pi_1^{\text{orb}}(B/\phi; r_1,\dots,r_m)= \left\langle \alpha_1, \dots, \beta_q, \gamma_1, \dots, \gamma_m \mid \textstyle\prod_{i=1}^q[\alpha_{i}, \beta_i]\textstyle\prod_{j=1}^m \gamma_j, \gamma_1^{r_1}, \dots, \gamma_m^{r_m}\right\rangle.\]
We call $(\rho(\gamma_1), \dots, \rho(\gamma_m))$ the ramification tuple of $\phi$; the Nielsen type is $(n_a)_{a\in \IZ/d}$ where $n_a$ is the number of $\gamma_j$ such that $\rho(\gamma_j) = a$.

Then the main result \cite[\"Aquivalenzsatz]{nielsen} can be formulated as follows: two automorphisms of the same order with the same ramification data are topologically equivalent if and only if they have the same Nielsen type. 

\begin{rem}
 The main step in \cite{nielsen} is to show that if $q>0$ then one can arrange that $\rho(\alpha_1) = 1$ and $\rho(\beta_1) = \rho(\alpha_2) = \dots = \rho(\beta_q) = 0$. He also goes on to show \cite[Equation~14.6]{nielsen} that the action of $\phi$ on $H_1(B;\IZ)$ has characteristic polynomial 
 \[\frac{(x^d-1)^{2q-2+m}(x-1)^2}{ (x^{n/r_1}-1)(x^{n/r_2}-1)\dots(x^{n/r_m}-1)}.\]
 This information is however not enough for the calculations of Section \ref{sect: computing}, because we reduce to torsion coefficients. 
\end{rem}

In applications, we are interested in the topological classification of configurations of two disjoint graphs of automorphism $\Gamma_{\phi_1}\cup \Gamma_{\phi_2}\subset B\times B$. Precomposing with one of the $\phi_i$ we can normalise one of the graphs to be the identity and the other one to be the graph of the fixed-point free automorphism $(\phi_1\circ \inverse\phi_2)^{\pm 1}$. Thus the classification of such configurations corresponds to the classification of topological types of fixed-point free automorphism up to taking the inverse, which gives the obvious action on Nielsen types. 

For each ramification type  in Proposition \ref{prop: ramification fpf autos} the possible Nielsen types can be easily analysed yielding the following result.
\begin{prop}\label{prop: fpf autos}
The ramification types of Prop.\ref{prop: ramification fpf autos} are uniquely realized by a Nielsen
type except for the following cases:
\begin{center}
 \begin{tabular}{cccl}
 \toprule
 genus $b$ & $d=\ord \phi$ & ramification type &ramification tuple \\
 \midrule
$9$ & $10$ & $(1\mid 5^2)$ & $(2,8)$, $(4,6)$\\
$9$ & $8$ & $(1\mid 2,4^2)$ & $(4,2,2)\sim(4,6,6)$\\ 
$9$ & $12$ & $(0 \mid 2,3^2, 4^2)$ & $(6,4,8,3,3)\sim(6,4,8,9,9)$\\
$9$&$10$ & $(0\mid 2^4, 5^2)$ & $(5,5,5,5,2,8)$, $(5,5,5,5,4,6)$\\
\midrule
$8$&$18$ & $(0\mid 2^2, 9^2)$ & $(9,9,2,16)$, $(9,9,4,14)$, $(9,9,8,10)$\\
$8$&$15$ & $(0\mid 3^2, 5^2)$ & $(5,10,3,12)$, $(5,10,6,9)$\\
$8$&$10$ & $(0\mid 2^2,5^3)$ & $(5,5,2,2,6)\sim(5,5,8,8,4)$, \\
& & & $(5,5,2,4,4)\sim(5,5,8,6,6)$\\
$8$&$6$ & $(0\mid 2^2,3^5)$ & $(3,3,2,2,2,2,4)\sim(3,3,2,4,4,4,4)$\\
\midrule
$7$ & $6$ & $(1\mid 3^3)$ & $(2,2,2)\sim(4,4,4)$\\
& & & $(4,9,9,2)\sim(8,3,3,10)$\\
$7$ & $6$ & $(0 \mid 2^4, 3^3)$ & $(3,3,3,3,2,2,2)\sim(3,3,3,3,4,4,4)$\\
\midrule
$6$&$14$ & $(0\mid 2^2, 7^2)$ & $(7,7,2,12)$, $(7,7,4,10)$, $(7,7,6,8)$\\
\midrule
$4$&$6$ & $(0\mid 2^2,3^3)$ & $(3,3,2,2,2)\sim(3,3,4,4,4)$\\
$4$&$10$ & $(0\mid 2^2, 5^2)$ & $(5,5,2,8)$, $(5,5,4,6)$\\
\bottomrule
 \end{tabular}
\end{center}
Ramification tuples related by $\sim$ correspond to topologically equivalent configurations of pairs
of automorphisms with disjoint graphs $\Gamma_\id \cup \Gamma_\phi\subset B\times B$.
\end{prop}


\begin{thebibliography}{EKKOS02}

\bibitem[Ati69]{atiyah69}
M.~F. Atiyah.
\newblock The signature of fibre-bundles.
\newblock In {\em Global {A}nalysis ({P}apers in {H}onor of {K}. {K}odaira)},
  pages 73--84. Univ. Tokyo Press, Tokyo, 1969.

\bibitem[BD02]{BD}
Jim Bryan and Ron Donagi.
\newblock Surface bundles over surfaces of small genus.
\newblock {\em Geometry \& Topology}, 6(1):59--67, 2002.

\bibitem[BDS01]{BDS}
Jim Bryan, Ron Donagi, and Andr{\'a}s~I Stipsicz.
\newblock Surface bundles: some interesting examples.
\newblock {\em Turkish Journal of Mathematics}, 25(1):61--68, 2001.

\bibitem[Bir16]{birman2016braids}
Joan~S Birman.
\newblock {\em Braids, Links, and Mapping Class Groups.(AM-82)}, volume~82.
\newblock Princeton University Press, 2016.

\bibitem[Bre00]{breuer}
Thomas Breuer.
\newblock {\em Characters and automorphism groups of compact {R}iemann
  surfaces}, volume 280 of {\em London Mathematical Society Lecture Note
  Series}.
\newblock Cambridge University Press, Cambridge, 2000.

\bibitem[Bro91]{Broughton:91}
S.~Allen Broughton.
\newblock Classifying finite group actions on surfaces of low genus.
\newblock {\em J. Pure Appl. Algebra}, 69(3):233--270, 1991.

\bibitem[CHS57]{chs57}
S.~S. Chern, F.~Hirzebruch, and J.-P. Serre.
\newblock On the index of a fibered manifold.
\newblock {\em Proc. Amer. Math. Soc.}, 8:587--596, 1957.

\bibitem[CLP11]{CLP11}
Fabrizio Catanese, Michael L{\"o}nne, and Fabio Perroni.
\newblock Irreducibility of the space of dihedral covers of the projective line
  of a given numerical type.
\newblock {\em Atti Accad. Naz. Lincei Cl. Sci. Fis. Mat. Natur. Rend. Lincei
  (9) Mat. Appl.}, 22(3):291--309, 2011.

\bibitem[CLP15]{clp15}
Fabrizio Catanese, Michael L{\"o}nne, and Fabio Perroni.
\newblock The irreducible components of the moduli space of dihedral covers of
  algebraic curves.
\newblock {\em Groups Geom. Dyn.}, 9(4):1185--1229, 2015.

\bibitem[CR09]{cat-roll09}
Fabrizio Catanese and S{\"o}nke Rollenske.
\newblock Double {K}odaira fibrations.
\newblock {\em J. Reine Angew. Math.}, 628:205--233, 2009.

\bibitem[Edm82]{Edmonds}
Allan~L. Edmonds.
\newblock Surface symmetry. {I}.
\newblock {\em Michigan Math. J.}, 29(2):171--183, 1982.

\bibitem[EKKOS02]{EKKOS}
H.~Endo, M.~Korkmaz, D.~Kotschick, B.~Ozbagci, and A.~Stipsicz.
\newblock Commutators, {L}efschetz fibrations and the signatures of surface
  bundles.
\newblock {\em Topology}, 41(5):961--977, 2002.

\bibitem[End98]{En}
Hisaaki Endo.
\newblock A construction of surface bundles over surfaces with non-zero
  signature.
\newblock {\em Osaka J. Math.}, 35(4):915--930, 1998.

\bibitem[GAP16]{GAP4}
The GAP~Group.
\newblock {\em {GAP -- Groups, Algorithms, and Programming, Version 4.8.4}},
  2016.

\bibitem[GDH91]{gdh91}
Gabino Gonz{\'a}lez~D{\'{\i}}ez and William~J. Harvey.
\newblock On complete curves in moduli space. {I}, {II}.
\newblock {\em Math. Proc. Cambridge Philos. Soc.}, 110(3):461--466, 467--472,
  1991.

\bibitem[GDH99]{gonzalez99}
G~Gonzalez-Diez and WJ~Harvey.
\newblock Surface groups inside mapping class groups.
\newblock {\em Topology}, 38(1):57--69, 1999.

\bibitem[GH81]{GH}
Marvin~J. Greenberg and John~R. Harper.
\newblock {\em Algebraic topology}, volume~58 of {\em Mathematics Lecture Note
  Series}.
\newblock Benjamin/Cummings Publishing Co., Inc., Advanced Book Program,
  Reading, Mass., 1981.
\newblock A first course.

\bibitem[Gol74]{goldberg74}
Charles~H Goldberg.
\newblock An exact sequence of braid groups.
\newblock {\em Mathematica Scandinavica}, 33(1):69--82, 1974.

\bibitem[GR58]{GR58}
Hans Grauert and Reinhold Remmert.
\newblock Komplexe {R}\"aume.
\newblock {\em Math. Ann.}, 136:245--318, 1958.

\bibitem[Hil02]{Hillman}
J.~A. Hillman.
\newblock {\em Four-manifolds, geometries and knots}, volume~5 of {\em Geometry
  \& Topology Monographs}.
\newblock Geometry \& Topology Publications, Coventry, 2002.

\bibitem[Hir69]{hirzebruch69}
F.~Hirzebruch.
\newblock The signature of ramified coverings.
\newblock In {\em Global {A}nalysis ({P}apers in {H}onor of {K}. {K}odaira)},
  pages 253--265. Univ. Tokyo Press, Tokyo, 1969.

\bibitem[JMSV]{MAPCLASSpackage}
A.~James, K.~Magaard, S.~Shpectorov, and H.~V\"olklein.
\newblock The gap package mapclass.

\bibitem[JY83]{jy83}
J{\"u}rgen Jost and Shing~Tung Yau.
\newblock Harmonic mappings and {K}\"ahler manifolds.
\newblock {\em Math. Ann.}, 262(2):145--166, 1983.

\bibitem[{Kas}68]{Kas}
A.~{Kas}.
\newblock {On deformations of a certain type of irregular algebraic surface.}
\newblock {\em {Am. J. Math.}}, 90:789–804, 1968.

\bibitem[Kod67]{kodaira67}
K.~Kodaira.
\newblock A certain type of irregular algebraic surfaces.
\newblock {\em J. Analyse Math.}, 19:207--215, 1967.

\bibitem[Kot99]{kotschick99}
D.~Kotschick.
\newblock On regularly fibered complex surfaces.
\newblock In {\em Proceedings of the {K}irbyfest ({B}erkeley, {CA}, 1998)},
  volume~2 of {\em Geom. Topol. Monogr.}, pages 291--298 (electronic). Geom.
  Topol. Publ., Coventry, 1999.

\bibitem[Lee17]{Lee}
Ju~A Lee.
\newblock Surface bundles over surfaces with a fixed signature.
\newblock {\em J. Korean Math. Soc.}, 54(2), 2017.

\bibitem[{Mey}73]{Meyer}
Werner {Meyer}.
\newblock {Die Signatur von Fl\"achenb\"undeln.}
\newblock {\em {Math. Ann.}}, 201:239–264, 1973.

\bibitem[MV08]{mitsumatsu2008foliations}
Yoshihiko Mitsumatsu and Elmar Vogt.
\newblock Foliations and compact leaves on 4-manifolds {I}: Realization and
  self-intersection of compact leaves.
\newblock {\em Groups of Diffeomorphisms: In Honor of Shigeyuki Morita on the
  Occasion of his 60th Birthday}, pages 415--442, 2008.

\bibitem[{Nie}37]{nielsen}
J.~{Nielsen}.
\newblock {Die Struktur periodischer Transformationen von Fl\"achen.}
\newblock {Danske Vidensk. Selsk. Math.-fys. Medd. 15, Nr. 1, 77 s (1937).},
  1937.

\bibitem[Smi99]{smith99}
Ivan Smith.
\newblock Lefschetz fibrations and the {H}odge bundle.
\newblock {\em Geom. Topol.}, 3:211--233 (electronic), 1999.

\bibitem[Wei16]{Weigl}
Sascha Weigl.
\newblock {\em Irreducible components of the space of curves with split
  metacyclic symmetry}.
\newblock PhD thesis, Bayreuth University, 2016.

\bibitem[Zaa95]{zaal95}
Chris Zaal.
\newblock Explicit complete curves in the moduli space of curves of genus
  three.
\newblock {\em Geom. Dedicata}, 56(2):185--196, 1995.

\end{thebibliography}

 \end{document}